\documentclass[12pt, reqno]{amsart}

\makeatletter
\g@addto@macro{\endabstract}{\@setabstract}
\makeatother

\usepackage{graphicx, stackrel}
\usepackage{amsmath, amssymb, amsthm, mathtools}
\usepackage{verbatim}
\usepackage{amsfonts}
\usepackage[authoryear,round]{natbib}
\usepackage{fancyvrb}
\usepackage{color}
\usepackage[inline]{enumitem}
\setlist[enumerate,1]{label=(\arabic*),ref=(\arabic*)}

\usepackage[citecolor=blue, colorlinks=true, linkcolor=blue]{hyperref}
\usepackage{doi}
\usepackage{mathrsfs}
\usepackage{bbm}

\usepackage[left=1.25in, right=1.25in, top=1.0in, bottom=1.15in, includehead, includefoot]{geometry}

\usepackage{xspace}
\newcommand{\iid}{\textsc{iid}\xspace} 
\newcommand{\set}[1]{\left\{ #1 \right\}} 
\newcommand{\abs}[1]{\left\lvert #1 \right\rvert} 
\newcommand{\norm}[1]{\left\lVert #1 \right\rVert} 
\newcommand{\re}{\hat{r}}
\newcommand{\ke}{\hat{\kappa}}

\renewcommand{\leq}{\leqslant}
\renewcommand{\geq}{\geqslant}

\DeclareMathOperator*{\argmax}{argmax}

\newcommand{\maximize}{\operatorname{maximize}}
\newcommand{\st}{\operatorname{subject~to}}

\newcommand{\iidsim}{\stackrel {\textrm{ {\sc iid }}} {\sim} }

\newcommand*\diff{\mathop{}\!\mathrm{d}}
\renewcommand{\epsilon}{\varepsilon}

\newcommand{\me}{\mathrm{e}}

\newcommand{\sB}{\mathscr{B}}

\newcommand{\cG}{\mathcal{G}}

\newcommand{\cV}{\mathcal{V}}

\newcommand{\XX}{\mathsf X}
\renewcommand{\AA}{\mathsf A}

\newcommand{\DD}{\mathsf D}
\newcommand{\ZZ}{\mathsf Z}
\newcommand{\YY}{\mathsf Y}

\newcommand{\RR}{\mathbbm R}

\newcommand{\NN}{\mathbbm N}

\newcommand{\EE}{\mathbbm E} 

\theoremstyle{plain}
\newtheorem{theorem}{Theorem}[section]

\newtheorem{lemma}[theorem]{Lemma}

\theoremstyle{definition}

\newtheorem{example}{Example}[section]
\newtheorem{remark}{Remark}[section]
\newtheorem{assumption}{Assumption}[section]

\begin{document}

\title{}

\date{\today}

\begin{center}
	\Large 
    Unbounded Dynamic Programming via the Q-Transform\footnote{We
        thank Takashi Kamihigashi and Yiannis Vailakis for valuable feedback
        and suggestions, as well as audience members at the Econometric
        Society meeting in Auckland in 2018 and the 2nd Conference on
        Structural Dynamic Models in Copenhagen.  Financial supports from ARC
        Discovery Grant
		DP120100321 and NSFC No.72003138 are gratefully acknowledged. \\ 
		\emph{Email addresses:} \texttt{qingyin.ma@cueb.edu.cn}, 
    \texttt{john.stachurski@anu.edu.au}, 
    \texttt{atoda@ucsd.edu} }

	\bigskip
	\normalsize
	Qingyin Ma\textsuperscript{a}, John Stachurski\textsuperscript{b}, 
    Alexis Akira Toda\textsuperscript{c} \par \bigskip
	
	\textsuperscript{a}ISEM, Capital University of Economics and Business  \\
	\textsuperscript{b}Research School of Economics, Australian National University \\
	\textsuperscript{c}Department of Economics, University of California San Diego \bigskip
	
	\today
\end{center}

\begin{abstract}
    We propose a new approach to solving dynamic decision problems with
    unbounded rewards based on the transformations used in Q-learning.
    In our case, the 
    objective of the transform is to convert an unbounded dynamic program into 
    a bounded one. The approach is general enough to handle problems for which 
    existing methods struggle, and yet simple relative to other techniques and accessible for applied work.
    We show by example that many common decision problems satisfy our conditions.

\end{abstract}


\section{Introduction}

Dynamic programming forms the backbone of modern economics. %
  Every year, thousands of students in graduate programs
around the world learn the standard methodology for infinite horizon problems with
discounting. Constructed primarily by \cite{blackwell1962discrete,
blackwell1965discounted}, this theory uses contraction mappings over spaces of
bounded functions metrized by the supremum norm. The approach is elegant,
powerful in terms of deriving theoretical results and, when applicable,
generates globally convergent algorithms. Standard textbook
treatments can be found in \cite{stokey1989} and \cite{bertsekas2017dynamic}.

Unfortunately, the approach is not actually applicable in the vast majority of concrete
economic problems.  This is due to the fact that almost all reward functions used in
applications are unbounded. For example, in quantitative work, the most
commonly used flow utility function is the
constant relative risk aversion (CRRA) specification
\begin{equation}\label{eq:CRRA}
    u(x)=
    \begin{cases*}
        \frac{x^{1-\gamma}}{1-\gamma} & if $\gamma>0$ and $\gamma\neq 1$,\\
        \log x & if $\gamma=1$,
    \end{cases*}
\end{equation}
where $\gamma>0$ is the risk aversion coefficient.  The function $u$ is
unbounded above if $0<\gamma<1$, is unbounded below if $\gamma>1$, and is
unbounded both from above and below if $\gamma=1$. Unbounded reward functions
violate Blackwell's conditions.


The need to deal with unbounded reward functions has led researchers to build
various extensions of Blackwell's theory. These extensions typically involve
either (a) recovering contractivity by modifying the metric that
measures distance between candidate value functions, or (b) introducing a
weaker form of contractivity that preserves at least some of Blackwell's
optimality results.  The former approach is exemplified by the weighted
supremum norm method, introduced by \cite{wessels1977markov} and applied
to economic problems by \cite{boyd1990recursive}, \cite{alvarez1998dynamic}, 
\cite{bauerle2018stochastic} and several other authors.\footnote{Modern
    summaries of the method can be found in
\cite{Hernandez-LermaLasserre1999}, \cite{bauerle2011markov} and
\cite{bertsekas2013abstract}.} The second approach can be seen in the work of
\cite{rincon2003existence} and \cite{martins2010existence} for the
deterministic case and \cite{matkowski2011discounted} for the stochastic case,
who apply local contractions on successively larger subsets of the state
space.

While these techniques are ingenious, and certainly important from a
theoretical perspective, their direct impact on quantitative applications
in economics has, as yet, been limited.  Weighted supremum norms work well with certain
problems but struggle with others, such as when rewards are unbounded 
below.\footnote{\label{fn:ubdd_below}To be specific, the weighted supremum 
	norm approach can be applied broadly when the reward function 
	is unbounded above and bounded below. However, when the reward 
	function is unbounded below, it is generally hard to construct 
	a proper weighting function and thus a contraction mapping via 
	the weighted supremum norm. Therefore, it is challenging to characterize
	the value function as the unique fixed point of the Bellman operator.
	For more background, see, e.g., \cite{le2005recursive} and 
	\cite{jaskiewicz2011discounted}.}
Local contraction methods are broadly applicable under reasonable assumptions
but require verifying technical conditions involving increasing
sequences of compact sets that exhaust the state space.\footnote{It is required that
the decision problem satisfies contractivity on each of these compact sets,
and some control over the way that contractivity fades must also be imposed
(see, e.g., \cite{matkowski2011discounted}, Assumptions~A1--A5, C1-C2 and
D1-D2).} Proofs of convergence properties are significantly more complex than
the bounded case, and the statements of the theorems are more challenging to
interpret.

In this paper we take an alternative route.  Rather than transforming the
standard contraction mapping theory of Blackwell to handle unbounded dynamic
programs, we transform the unbounded dynamic programs into bounded ones so
that standard contraction mapping theory can be applied.  The transformation
that we use maps value functions into the ``action-value'' functions 
used in Q-learning, a popular reinforcement learning algorithm that
allows online updating by a controller in an incremental fashion (see, e.g., 
\cite{watkins1992q} or \cite{szepesvari2010algorithms}).
It has been shown that the algorithm has strong global convergence 
properties and our results are in this spirit.

The core idea is as follows: In standard dynamic
programs, the Bellman operator is defined by composing the following two
operations: given a candidate value function $v$,
\begin{enumerate*}[label=(\alph*)]
	\item\label{item:ope1}
	compute the discounted expectation $g\coloneqq \beta\EE v$ over current
	states and actions, and 
	\item\label{item:ope2}
	add current reward $r$, maximize over current feasible actions, and then
	update the candidate value function as $v=\max\set{r+g}$.
\end{enumerate*}
%
Instead of aiming at updating $v$, 
one can transform the Bellman operator by applying operations
$\ref{item:ope2} \to \ref{item:ope1}$, which is
the Q-transform, and focus on updating the candidate 
``action-value'' function $g$.
 We show that, under relatively weak and 
easily testable conditions, the transformed Bellman operator is a 
contraction with unique fixed point $g^*$ that is bounded, and the 
true value function $v^*$ can be recovered as $v^*=\max\set{r+g^*}$.


One advantage of the transformation-based approach to 
unbounded programs adopted in this paper is that the methodology fits 
well with the case where the weighted supremum norm approach struggles: 
maximization problems where rewards are unbounded below (see 
Footnote~\ref{fn:ubdd_below}). Such optimization problems are commonplace 
in quantitative applications, such as those involving CRRA flow utility 
with $\gamma > 1$, which is the empirically relevant case.
We show that, for many canonical applications from this class of problems, the 
Q-transform converts unbounded value functions into bounded action-value functions.  
Standard contraction mapping theory can then be applied.

A second advantage of the Q-transform method is that it has no difficulty handling
stochastic dynamic programs. In fact, the action-value functions associated
with the Q-transform tend to be well-behaved and regular in the stochastic case,
due to the fact that conditional expectations operators have a smoothing
effect on functions. Therefore, it requires less restriction on the primitive setup
compared with many existing methods. In contrast, many existing papers on unbounded 
dynamic programming either focus on the deterministic case or restrict the support 
of shocks.\footnote{\cite{alvarez1998dynamic} handle certain homogeneous problems
	using weighted supremum norm methods, although they focus on the deterministic case.
	A generalization to the stochastic case requires bounds on the maximum growth rate.
	Assumptions~D1-D2 of \cite{matkowski2011discounted} also require the state not to
	jump too much. These assumptions are strong from an applied perspective.}

A third advantage is that, despite its generality, the Q-transform
approach is accessible to a general audience. In particular, 
it relies mainly on the standard contraction mapping theorem, and can be 
conveniently generalized to handle dynamic programs where rewards are both 
unbounded above and unbounded below. We provide many canonical examples 
such as optimal savings, optimal default, job search, and optimal portfolio,
all with unbounded rewards and in stochastic environments.  



On a technical level, the contribution of our paper is twofold. First, we
identify general sufficient conditions under which unbounded dynamic programs
can be transformed into bounded ones.  Second, we prove that, when such a
transformation is available, the solution to the transformed problem is equal
to that to the original problem. To the best of our knowledge, this is the
first research in which the Q-transform has been used to convert
unbounded reward dynamic programs into bounded ones.\footnote{Researchers in
    economics have used alternative transformations of the Bellman equation
    when studying dynamic programming problems,
including \cite{rust1987optimal}, \cite{jovanovic1982selection},  
\cite{abbring2018very} and \cite{ma2021dynamic}.
These transformations are typically aimed at improving economic intuition,
estimation properties or computational efficiency.}

There are connections between our work and the study of unbounded dynamic
programming in \cite{kamihigashi2014elementary}. Assuming the Bellman operator
maps an order interval of functions into itself and some transversality-like
conditions hold, \cite{kamihigashi2014elementary} shows the existence and
uniqueness of the fixed point of the Bellman equation and obtains optimality
properties. The relative advantages of
the approach presented here include treating stochastic decision
problems (\cite{kamihigashi2014elementary} restricts attention to the
deterministic case) and obtaining uniform geometric rates of convergence in
value function iteration, rather than pointwise convergence. 

Our work is also related to results in \cite{van1980stochastic} and
\cite{jaskiewicz2011discounted}, which explicitly admit problems with rewards
that are unbounded below.  In this setting, \cite{jaskiewicz2011discounted}
show that the value function of a Markov decision process is a solution to the
Bellman equation.  The methodology developed here strengthens their results by
adding uniqueness and proving that value function iteration leads to an
optimal policy. In an extension section, we combine our
methodology with the weighted supremum norm approach, allowing us to handle
problems that are both unbounded above and unbounded below.

Some studies have approached dynamic programming with unbounded
rewards via an Euler equation method, as seen for example in
\cite{kuhn2013recursive} and \cite{MaStachurskiToda2020JET}.  This methodology can
be powerful but is limited in scope.  For example, in Section~\ref{s:app}, we
show how the Q-transform method can be applied to the kinds of
optimal savings problem with endogenous labor choice that are common in
both theoretic and applied works (see, e.g.,
\cite{CastanedaDiazGimenezRiosRull2003} and \cite{zhu2020existence}).  
The Euler equation method of
\cite{kuhn2013recursive} and \cite{MaStachurskiToda2020JET} is not applicable in
this setting because the choice variable (consumption and labor) is
multi-dimensional. Similarly, the Euler equation method is not
applicable in the optimal savings problem in Section~\ref{ss:portfolio}, due to
nontrivial portfolio choice. Optimal consumption-portfolio problems have been
mostly studied in the literature under special homogeneity assumptions
\citep{samuelson1969,Toda2014JET} or finite horizon \citep{HePearson1991}. Our
framework shows that the problem can be studied in an infinite-horizon
environment when the utility function is unbounded below.

The Q-transform is not limited to dynamic programs that are
additively separable. In a recent paper,
\cite{bauerle2018stochastic} study an optimal growth model in the presence of
risk-sensitive preference, in which the agent is risk averse in future utility
(in addition to being risk averse in future consumption).\footnote{This is in
    comparison with the classical additively separable preference model, where
    the agent is risk neutral in future utility.  The additively separable
    preference is a special limiting case by letting $\gamma\to 0$ and using
    the property $\lim_{\gamma\to 0}-\frac{1}{\gamma}\log \EE \me^{-\gamma
    X}=\EE X$ for a random variable $X$.  Further comments on risk-sensitive
preference can be found in \cite{follmer2004stochastic},
\cite{bauerle2011markov}, \cite{bauerle2018stochastic} and references cited
therein.  Models with risk-sensitive preference are also related to robust
control problems, as discussed in \cite{hansen2008robustness}.} They provide
valuable optimality results, although these results cannot treat many common
period utility functions, such as CRRA with relative risk aversion at least  
one or logarithm utility,
because they exclude all utility functions that are unbounded below.
Furthermore, an optimal growth model is a rather special dynamic program.
In an extension section, we present a general theory of dynamic programming with
risk-sensitive preferences via the Q-transform.

The rest of our paper is structured as follows.  Section~\ref{s:ea} starts the
exposition with typical examples. Section~\ref{s:gf} presents the general
theory when rewards are bounded above (though potentially unbounded below).
Section~\ref{s:app} provides additional applications.  Section~\ref{s:ext}
extends the general theory to the case when rewards are unbounded both from
above and below, and also considers the case with recursive (risk-sensitive) 
preferences. Section~\ref{s:conc} concludes. Main proofs are deferred to the 
appendix.

\section{Example Applications}\label{s:ea}

We first illustrate the methodology for converting unbounded
problems to bounded ones in some relatively simple settings.
More sophisticated applications are deferred to Section~\ref{s:app} 
after presentation of the theory.

\subsection{Application 1: Optimal Savings}\label{ss:os}

Consider an optimal savings problem where a borrowing constrained agent solves
\begin{subequations}\label{eq:os}
\begin{align}
&\maximize && \EE \sum_{t = 0}^\infty \beta^t u(c_t)\label{eq:os_obj}\\
&\st &&0 \leq c_t \leq w_t,\label{eq:os_borrow} \\
&&& w_{t+1} = R (w_t - c_t) + y_{t+1},\label{eq:os_budget}
\end{align}
\end{subequations}
with $(w_0,y_0)$ given. Here $\beta \in (0,1)$ is the discount factor,
$c_t, w_t, y_t\geq 0$ are, respectively, consumption, wealth and non-financial
income at time $t$, $R\geq 0$ is the gross rate of return on financial
income, and $u \colon \RR_+\to \RR\cup\set{-\infty}$ is a utility function, 
which is increasing and continuous.\footnote{\label{fn:timing}The optimal
	savings problem represents one of the fundamental workhorses of modern
	macroeconomics. By convention, the financial wealth $w_t$ in
	the budget constraint in \eqref{eq:os_budget} includes the current
	non-financial income $y_t$. One can modify the budget constraint to an
	alternative timing such as $w_{t+1} = R (w_t - c_t + y_t)$, where the 
	time $t$ financial wealth $w_t$ excludes current income $y_t$, and the 
	arguments below still go through after suitable modifications. An 
	application along these lines is given in Section \ref{ss:os_labor}.} 
For now, suppose that $u$ is bounded above but 
unbounded below, with $u(0)=-\infty$. This is the case for, say, the constant
relative risk aversion (CRRA) specification \eqref{eq:CRRA} with 
$\gamma>1$ (as in much of the literature).

Assume that $\set{y_t}$ satisfies $y_t=y(z_t,\xi_t)$, where $z_t$ is a Markov
process with state space $\ZZ$, $\xi_t$ is an \iid shock of arbitrary
dimension, and $y$ is a nonnegative measurable function. The Bellman equation
of this problem is
\begin{equation}\label{eq:be}
    v(w, z) = \sup_{0 \leq c \leq w}\set{u(c) + \beta \EE_z v(R (w - c) + y', z')},
\end{equation}
where $y'=y(z',\xi')$.  The value function
is unbounded below, and the classical arguments in
\cite{blackwell1965discounted} cannot be applied.\footnote{To confirm this,
suppose to the contrary that $v$ is the value function and $\abs{v} \leq M <
\infty$. Then $-M \leq v(0, z) \leq u(0) + \beta M = -\infty$. Contradiction.} 

Consider, however, the following line of argument.  Suppose that 
\begin{equation}\label{eq:Ezuy}
    \inf_z \EE_z u(y(z',\xi'))
    >-\infty,
\end{equation}
which is a relatively mild restriction.\footnote{The expectation in \eqref{eq:Ezuy} 
	should be understood as $\EE[u(y(z_{t+1},\xi_{t+1})) \mid z_t=z]$. 
	Condition~\eqref{eq:Ezuy} holds if, say, $y(z,\xi) = z$ for all $\xi$ and $\ZZ$
    is finite and positive \citep{aiyagari1994,Cao2020}, 
    or if income has a persistent-transitory representation \citep{heathcote2010macroeconomic,EjrnaesBrowning2014} such as 
    $\log y(z,\xi)=\mu(z)+\sigma(z)\xi$, with suitable distributional assumptions 
    (e.g., $u$ is CRRA, $z$ is a finite state Markov chain, and $\xi$ has a finite
    moment generating function).} 
Let
\begin{equation}\label{eq:gfv}
    g(w,z,c) \coloneqq \beta \EE_z v(R(w-c) + y', z').
\end{equation}
The function $g$ is called the \emph{action-value function}, since it returns
the value of the state after committing to a given action in the current
period (and using continuation values dictated by $v$ thereafter). Following the
spirit of Q-learning, we begin
by rewriting the Bellman equation in terms of the action-value function 
alone.\footnote{Our approach is slightly different from the Q-learning approach, 
	where the action-value function is defined as $u + g$ instead. We will show that
	eliminating the current reward (which can be unbounded below) can help us transform
	the action-value function into a bounded function and, as a result, the 
	classical dynamic programming theory applies. This is the essence of our
	approach.} 



In particular, we combine \eqref{eq:be} and \eqref{eq:gfv} to give
\begin{equation}\label{eq:nv}
v(w, z) = \sup_{0 \leq c \leq w}\set{u(c) + g(w,z,c)}.
\end{equation}
We eliminate the function $v$ from \eqref{eq:nv} by using the definition
of $g$ in \eqref{eq:gfv}.  The first step is to evaluate $v$ in \eqref{eq:nv}
at $(R(w-c) + y', z')$, which gives 
\begin{equation*}
v(R(w-c) + y', z') = \sup_{0 \leq c' \leq R(w-c) + y'}\set{u(c') + g(R(w-c) + y',z',c')}.
\end{equation*}
Taking the conditional expectation of both sides with respect to $z$,
multiplying by $\beta$ and using \eqref{eq:gfv} again, we get
\begin{equation}\label{eq:ng}
g(w,z,c) = \beta \EE_z \sup_{0 \leq c' \leq R(w-c) + y'}\set{u(c') + g(R(w-c) + y',z',c')},
\end{equation}
which is a functional equation in $g$. Consider a transformed Bellman operator
$S$ such that $Sg(w,z,c)$ is equal to the right hand side of \eqref{eq:ng}.
By construction, any solution $g$ of \eqref{eq:ng} is a fixed point of $S$ and
vice versa. Let $\cG$ be the space of bounded measurable functions on the set
$\DD$ defined by
\begin{equation*}
    \DD\coloneqq \set{(w,z,c)\in \RR_+ \times \ZZ \times \RR_+:c\leq w}
\end{equation*}
equipped with the supremum norm $\norm{\cdot}$. The set $\cG$ can be
understood as the family of candidate action-value functions for this problem. We claim that $S$ maps $\cG$
into itself and, moreover, is a contraction of modulus $\beta$ with respect to
the supremum norm.

To see that this is so, pick any $g \in \cG$. Then $Sg$ is bounded above,
since
\begin{equation*}
    Sg(w,z,c) \leq \beta (\sup u + \norm{g}) < \infty.
\end{equation*}
More importantly, $Sg$ is bounded below. Indeed, using $g(w',z',c')\geq
-\norm{g}$ and the monotonicity of $u$, we obtain
\begin{align*}
	Sg(w,z,c) 
	& \geq 
	\beta \EE_z
	\sup_{0 \leq c' \leq R(w-c) + y'}
	\set{u(c') - \norm{g}}
	\\
	& =
	\beta \EE_z
	\set{u(R(w-c) + y') - \norm{g}}\\
	& \geq 
	\beta \EE_z u(y') - \beta \norm{g}.
\end{align*}
The last term is finite by \eqref{eq:Ezuy}. Hence $S$ is a self map on $\cG$. To show
that $S$ is a contraction mapping, we verify \cite{blackwell1965discounted}'s
sufficient conditions. From \eqref{eq:ng} we see that $g_1\leq g_2$
implies $Sg_1\leq Sg_2$, so monotonicity holds. If $M\geq 0$ is any constant,
then for any $g\in \cG$ we have
\begin{align*}
    S(g+M)(w,z,c)
    &=\beta \EE_z \sup_{0 \leq c' \leq R(w-c) + y'}\set{u(c') + g(R(w-c) + y',z',c')+M}\\
    &=\beta \EE_z \sup_{0 \leq c' \leq R(w-c) + y'}\set{u(c') + g(R(w-c) + y',z',c')}+\beta M\\
    &=Sg(w,z,c)+\beta M,
\end{align*}
so the discounting property holds.  We have now shown that $S$ is a
contractive self-map on $\cG$. Moreover, $\cG$ is a space of bounded
functions. By Banach's contraction mapping theorem, $S$ has a unique fixed
point $g^*$ in $\cG$. 


It is natural to guess that we can now insert $g^*$ into the right hand side
of \eqref{eq:nv}, 
maximize at each state, and obtain the optimal consumption policy.  We show
that this conjecture is correct and, more generally, that Bellman's principle
of optimality vis-\`a-vis the transformed Bellman equation also holds.  The
arguments are not trivial, since the transformation in \eqref{eq:gfv} that
maps $v$ to $g$ is not bijective.  Full details are provided in
Section~\ref{s:gf}.\footnote{For the savings problem treated above, one can
    also use Euler equation methods, which circumvent some of the issues
    associated with unbounded rewards (see, e.g., \cite{kuhn2013recursive} and
    \cite{MaStachurskiToda2020JET}).  However, for many practical
    applications, these Euler equation arguments cannot be used, due to
    features such as recursive preferences or discrete or multi-dimensional
    choices (see below).  Moreover, our detailed treatment of optimal savings
    in Section~\ref{ss:os_cont} shows that, even when Euler equation methods
    are available, the assumptions needed for the theory in this paper are
significantly weaker, at least in some dimensions.}

\subsection{Application 2: Optimal Default}

\label{ss:are}

Consider an infinite horizon optimal savings problem with default, in the
spirit of \cite{arellano2008default} and a large related
literature.\footnote{See, e.g., \cite{hatchondo2009heterogeneous}, 
	\cite{hatchondo2016debt} and \cite{AAHW2019}.} 
A country with current assets $w_t$ chooses between continuing to 
participate in international financial markets and defaulting. As in 
Section~\ref{ss:os}, output $y_t = y(z_t, \xi_t)$ is a function of 
a Markov process $\set{z_t}$ and an \iid shock $\set{\xi_t}$. 
To simplify the exposition, we assume that default leads to permanent
exclusion from financial markets, with lifetime value
\begin{equation*}
	v^d(y, z) = \EE \sum_{t = 0}^\infty \beta^t u(y_t).
\end{equation*}
The utility function $u$ has the same properties as Section~\ref{ss:os}.
The value of continued participation in financial markets is 
\begin{equation*}
	v^c(w, y, z)
	= \sup_{-b \leq w' \leq R(w + y)}
	\set{
	u(w + y - w'/R) + \beta \EE_{z} \, v(w', y', z')
	},
\end{equation*}
where $b > 0$ is a constant borrowing
constraint and $v$ is the value function satisfying
\begin{equation*}
	v(w, y, z)
	= \max
	\set{
	v^d(y, z)
	,\, 
	v^c(w, y, z)
	}.
\end{equation*}
The function $v$ is unbounded below because $u(0)=-\infty$.
However, we can convert this into a bounded problem, as the
following analysis shows. 

Let $i$ be a discrete choice variable taking values in $\set{0,1}$, with $0$
indicating default and $1$ indicating continued participation. We introduce
the action-value function
\begin{equation*}
g(z,w',i) \coloneqq 
	\begin{cases*}
	\beta \EE_{z} v^d(y', z') & if $i=0$,  \\
	\beta \EE_{z} v(w', y', z') & if $i=1$,
	\end{cases*}
\end{equation*}
so that for $-b \leq w' \leq R(w + y)$, we have
\begin{equation*}
	v(w, y, z)
	= \max
	\set{u(y) + g(z,w',0),\sup_{w'}\set{
	u(w + y - w'/R) + g(z, w',1)}}.
\end{equation*}
Eliminating the value function $v$ yields
\begin{equation*}
    g(z,w',0) = \beta \EE_{z} \set{u(y') + g (z',w',0)}
    \quad \text{and}
\end{equation*}
\begin{equation*}
	g(z, w',1) 
	= \beta \EE_{z}
	\max
	\set{
	u(y') + g(z',w',0) 
	, \,
	\sup_{w''}
	\set{
	u(w' + y' - w''/R) + g(z',w'',1)
	}
	}, 
\end{equation*}
%
where $-b \leq w'' \leq R(w' + y')$. We can then define the fixed point
operator $S$ corresponding to these functional equations.

If $g$ is bounded above by some constant $K$, then $Sg \leq \sup_c u(c) +
K$.  More importantly, if $g$ is bounded below by some constant
$M$, we obtain
\begin{align*}
    Sg(z,w',0) & \geq \beta \EE_{z} u(y') + \beta M,\\
	Sg(z, w',1)
	& \geq
	\beta \EE_{z}
	\max
	\set{
	u(y') + M 
	,\,
	u(w' + y' + b/R) + M
	}
	\\
	& =
	\beta \EE_{z}
	\max
	\set{
	u(y') 
	,\,
	u(w' + y' + b/R)
	}
	+ \beta M.
\end{align*}
Hence, $Sg$ is bounded below by a finite constant if \eqref{eq:Ezuy} holds. An
argument similar to the one in Section~\ref{ss:os} now proves that $S$ is a
contraction with respect to the supremum norm. 
(Section~\ref{ss:od_cont} gives details.)

\subsection{Application 3: Job Search}\label{ss:jsp}

Following \cite{mccall1970economics}, consider a search problem where an
unemployed worker can either accept the current job offer and work at that
wage forever or choose an outside option (e.g., work in the informal sector)
and continue to the next period. Letting $z_t$ be the worker's productivity at
time $t$, which is a Markov process, the job offer $w_t$ and outside option
$c_t$ satisfy
\begin{equation}\label{eq:zp}
    w_t=w(z_t,\xi_t)\quad \text{and}\quad c_t=c(z_t,\xi_t),
\end{equation}
where $w,c$ are nonnegative measurable functions and $\xi_t$ is an \iid shock
that could be vector-valued.\footnote{For instance, if the job offer $w_t$ and
outside option $c_t$ are independent conditional on $z_t$, then we may write
$\xi=(\xi_1,\xi_2)$, where $\xi_1$ and $\xi_2$ are independent, $w(z,\xi)$
depends only on $z$ and $\xi_1$, and $c(z,\xi)$ depends only on $z$ and
$\xi_2$.} Letting $u$ be the utility function and $\beta\in (0,1)$ be the
discount factor, a worker that accepts a job offer $w$ enjoys lifetime utility
$\sum_{t=0}^\infty \beta^t u(w)=\frac{u(w)}{1-\beta}$. Therefore, the worker's
value function satisfies the Bellman equation
\begin{equation}
	\label{eq:mcvf}
	v(w, c, z) = \max 
	\set{
	\frac{u(w)}{1-\beta},
	u(c) + \beta \EE_{z}v(w', c', z') 
	}.
\end{equation}
%
For now, let $u$ be bounded above. In addition, analogous to \eqref{eq:Ezuy}, assume
\begin{equation}\label{eq:Ezuy2}
    \text{either}~\inf_z\EE_zu(w')>-\infty~\text{or}~\inf_z\EE_zu(c')>-\infty.
\end{equation}
The value function $v(w,c,z)$ is unbounded below, if, say
$u$ is CRRA as in \eqref{eq:CRRA} with $\gamma>1$ and the job offer and
outside option in \eqref{eq:zp} can be arbitrarily small. To shift to a
bounded problem, we can proceed in a similar vein to our manipulation of the
Bellman equation in the optimal savings case. First we set
\begin{equation*}
	g(z) \coloneqq \beta \EE_{z} \, v(w', c', z'),
\end{equation*}
so that \eqref{eq:mcvf} can be written as 
\begin{equation*}
	v(w, c, z) = \max 
	\set{
	\frac{u(w)}{1-\beta},\;
	u(c) + g(z)
	}.
\end{equation*}
Next we use the definition of $g$ to eliminate $v$ from this last expression,
which leads to the functional equation
\begin{equation}
	\label{eq:hfe}
	g(z) = \beta \, \EE_{z} 
	\max 
	\set{
	\frac{u(w')}{1-\beta},\;
	u(c') + g(z')
	}.
\end{equation}
Let $S$ be an operator such that $Sg(z)$ is equal to the right hand side of
\eqref{eq:hfe}.  It is clear that if $g$ is bounded above then so is $Sg$.  In
addition, if $g$ is bounded below then so is $Sg$.  To show this, using the
elementary bound
\begin{equation}
\EE \max\set{X, Y} \geq \max\set{\EE X, \EE Y}\label{eq:EmaxXY}
\end{equation}
for arbitrary random variables $X, Y$ and $g\geq -\norm{g}$, we have
\begin{equation*}
Sg(z) \geq \beta \max \set{\EE_{z} \frac{u(w')}{1-\beta}, \EE_{z} u(c')-\norm{g}}.
\end{equation*}
Condition~\eqref{eq:Ezuy2} then implies that $Sg$ is also bounded below.

An argument similar to the one adopted in Sections~\ref{ss:os}--\ref{ss:are} 
shows that $S$ is a contraction mapping with respect to the
supremum norm on a space of bounded functions.  Thus, we can proceed down
the same path to establish optimality.

\section{General Formulation and Theory}
\label{s:gf}

Section~\ref{s:ea} showed how some unbounded problems can be converted to
bounded problems by transforming the Bellman equation. The next step is to
confirm the validity of such a transformation in terms of the connection
between the transformed Bellman equation and optimal policies. We do this in a
generic dynamic programming setting that contains the applications given
above. 

\subsection{Problem Formulation}
\label{ss:t}

For a given topological space $E$, let $\sB(E)$ be the Borel subsets of $E$.  For our
purpose, a dynamic program consists of
\begin{itemize}
	\item a nonempty set $\XX$ called the \emph{state space},
	\item a nonempty set $\AA$ called the \emph{action space},
    \item a nonempty correspondence $\Gamma:\XX \twoheadrightarrow \AA$ called
        the \emph{feasible correspondence}, along with the associated set of
        \emph{state action pairs}
	\begin{equation*}
	    \DD \coloneqq \set{ (x,a) \in \XX \times \AA : a \in \Gamma (x) },
	\end{equation*}
	\item a measurable map $r: \DD \to \RR \cup \set{-\infty}$ called the \emph{reward function},
    \item a constant $\beta \in (0,1)$ called the \emph{discount factor}, and 
    \item a \emph{stochastic kernel} $P$ governing the evolution of
        states.\footnote{Here a \emph{stochastic kernel} corresponding to our
            controlled Markov process $\set{(x_t,a_t)}$ is a mapping $P: \DD
            \times \sB(\XX) \to [0,1]$ such that
        \begin{enumerate*}
        \item for each $(x,a) \in \DD$, $A \mapsto P(x,a, A)$ is a probability measure on $\sB(\XX)$, and
        \item for each $A \in \sB(\XX)$, $(x,a) \mapsto P(x,a,A)$ is a measurable function.
        \end{enumerate*}}
\end{itemize}
Each period, an agent observes a state $x_t \in \XX$ and responds with an
action $a_t \in \Gamma(x_t) \subset \AA$. The agent then obtains a reward
$r(x_t, a_t)$, moves to the next period with a new state $x_{t+1}$, and
repeats the process by choosing $a_{t+1}$ and so on. The state process updates
according to $x_{t+1} \sim P (x_t, a_t, \cdot)$.

Let $\Sigma$ denote the set of \emph{feasible policies}, which we assume to
be nonempty and define as all measurable maps $\sigma: \XX \to \AA$ satisfying
$\sigma (x) \in \Gamma(x)$ for all $x \in \XX$.\footnote{We can and do focus
on stationary Markov policies in what follows since the value of any
nonstationary policy can be obtained by a stationary Markov policy.  See,
e.g., \citet[Section~2.1]{bertsekas2013abstract}.}
Given any policy $\sigma \in
\Sigma$ and initial state $x_0 = x \in \XX$, the \emph{$\sigma$-value
function} $v_\sigma$ is defined by
\begin{equation}\label{eq:vsigma}
	v_\sigma (x) = \sum_{t=0}^\infty \beta^t \EE_x r(x_t, \sigma(x_t))
\end{equation}
whenever the expectation and infinite sum are well-defined. We understand
$v_\sigma (x)$ as the lifetime value of following policy $\sigma$ now and
forever, starting from current state $x$.  

The \emph{value function} associated with this dynamic program is defined at each $x \in \XX$ by 
\begin{equation}\label{eq:szvf}
v^*(x) = \sup_{\sigma \in \Sigma} v_{\sigma} (x).
\end{equation}
A feasible policy $\sigma^*$ is called \emph{optimal} if $v_{\sigma^*} =
v^*$ on $\XX$. The objective of the agent is to find an optimal policy that
attains the maximum lifetime value. 

The Bellman equation associated with the dynamic program is
\begin{equation}\label{eq:bellman}
    v(x)=\sup_{a\in \Gamma(x)}\set{r(x,a)+\beta \EE_{x,a}v(x')},
\end{equation}
where $\EE_{x,a}$ denotes the expectation with respect to the probability
measure $P(x,a,\cdot)$. 

\subsection{The Q-Transform}
\label{ss:q-trans}
As in the examples in Section~\ref{s:ea}, we define the action-value 
function $g$ as $g(x,a)\coloneqq \beta \EE_{x,a}v(x')$. Similar to the
Q-learning approach, we rewrite the Bellman equation in terms of 
the action-value function, which gives
\begin{equation*}
    v(x)=\sup_{a\in \Gamma(x)}\set{r(x,a)+g(x,a)}.
\end{equation*}
Changing $(x,a)$ to $(x',a')$, multiplying both sides by $\beta$, taking the
conditional expectation with respect to $(x,a)$, and using the definition of
$g$, we obtain the transformed Bellman equation 
\begin{equation}\label{eq:tbe}
    g(x,a)=\beta\EE_{x,a}\sup_{a'\in \Gamma(x')}\set{r(x',a')+g(x',a')}.
\end{equation}

Motivated by this derivation, given a real-valued measurable function $g$ on
$\DD$, we define the \emph{transformed Bellman operator} $S$ by
\begin{equation}\label{eq:rfba}
    S g(x,a) \coloneqq \beta \EE_{x,a} \sup_{a' \in \Gamma(x')}
        \set{ r(x', a') + g(x',a') }.
\end{equation}
A feasible policy $\sigma$ is called \emph{$g$-greedy} if
\begin{equation}\label{eq:greedy}
    \sigma(x)\in \argmax_{a \in \Gamma(x)} \set{ r(x,a) + g(x,a) }
    \quad \text{for all $x \in \XX$.}
\end{equation}
At each $x \in \XX$ and $(x,a) \in \DD$, we define 
\begin{equation}\label{eq:rbar}
    \bar r (x) \coloneqq \sup_{a \in \Gamma (x)} r (x,a)
    \quad \text{and} \quad
    \re (x,a) \coloneqq \EE_{x,a} \bar r (x').
\end{equation}
The function $\bar{r}$ can be interpreted as the maximum reward given the
current state $x\in \XX$. The function $\re$ can be interpreted as its
expectation conditional on the previous state and action. We make the
following assumption.

\begin{assumption}\label{a:rbar}
    The function $\bar{r}$ in \eqref{eq:rbar} is bounded above and $\re$ is bounded below.
\end{assumption}



Let $\cG$ be the set of bounded measurable functions on $\DD$ and
$\norm{\cdot}$ be the supremum norm. In spite of the potentially unbounded
below rewards, the following result illustrates that $S$ maps elements of
$\cG$ into itself and the dynamic program can be solved via the standard
contraction mapping theorem.

\begin{theorem}
	\label{t:cs0}
	If Assumption~\ref{a:rbar} holds, then $v^*$ in \eqref{eq:szvf} is well-defined,
	\begin{enumerate}
		\item $S \cG \subset \cG$ and $S$ is a contraction mapping on $(\cG, \norm{\cdot})$,
		\item $S$ admits a unique fixed point $g^*$ in $\cG$, and
		\item $S^k g$ converges to $g^*$ at rate $O(\beta^k)$ under $\norm{\cdot}$.
	\end{enumerate}
	Moreover, if there exists a closed subset $\cG_1$ of $\cG$ such that $S \cG_1
	\subset \cG_1$ and a $g$-greedy policy exists for each $g \in \cG_1$, then
	\begin{enumerate}[label=(\alph*)]
		\item $g^*$ is an element of $\cG_1$ and satisfies
		\begin{equation*}
		g^*(x,a) = \beta \EE_{x,a} v^*(x') 
		\quad \text{and} \quad
		v^* (x) = \sup_{a \in \Gamma(x)} \set{ r(x,a) + g^*(x,a) },
		\end{equation*}
		\item at least one optimal policy exists, and
		\item a feasible policy is optimal if and only if it is $g^*$-greedy.
	\end{enumerate}
\end{theorem}

\subsection{Existence of Optimal Policy}
\label{ss:exist}

Theorem~\ref{t:cs0} states that under Assumption~\ref{a:rbar}, which is 
satisfied in many applications, the transformed Bellman operator $S$ is a 
contraction. However, it requires a high-level assumption to guarantee that 
a solution to the dynamic program exists.

We now discuss some general sufficient conditions for 
parts~\ref{item:cs.a}--\ref{item:cs.c} of Theorem~\ref{t:cs0} to hold.
To this end, we introduce an additional assumption.

\begin{assumption}\label{a:sc}
    \begin{enumerate*}
        \item The sets $\XX$ and $\AA$ are complete separable metric spaces,
        \item the reward function $r$ is upper semicontinuous,
        \item the feasible correspondence $\Gamma$ is compact-valued and upper hemicontinuous,%
        \footnote{In other words, the set $\set{x \in \XX: \Gamma(x)\subset
            U}$ is open for each open subset $U\subset \AA$. See \citet[Lemma
            17.4]{AliprantisBorder2006} for
            alternative characterizations of upper hemicontinuity.} and
        \item the stochastic kernel $P$ is Feller.\footnote{In other words,
            $(x,a) \mapsto \int h(x') P(x,a,\diff x')$ is bounded and
            continuous whenever $h$ is.} 
    \end{enumerate*}
\end{assumption}

In most applications of interest, Assumption ~\ref{a:sc} is satisfied.

Let $\cG_1$ be the set of upper semicontinuous functions in $\cG$. The
following theorem shows that the conclusions of Theorem~\ref{t:cs0} hold.

\begin{theorem}
	\label{t:exist0}
	If Assumptions~\ref{a:rbar} and \ref{a:sc} hold, then 
	$\cG_1$ is a closed subset of $\cG$, $S \cG_1 \subset \cG_1$, and a $g$-greedy policy exists for each $g \in \cG_1$.
	Consequently, all the conclusions of Theorem~\ref{t:cs0} hold and $g^*,v^*$ are upper semicontinuous.
\end{theorem}

\section{Applications}\label{s:app}

Now we complete the discussion of all applications in Section~\ref{s:ea}. We
also provide additional applications to optimal savings with endogenous labor
choice and optimal consumption-portfolio choice.

\subsection{Optimal Savings (Continued)}
\label{ss:os_cont}

Recall the optimal savings problem of Section~\ref{ss:os}. Following the
setting in \cite{MaStachurskiToda2020JET}, we allow for capital income risk in
a Markov environment. The agent seeks to solve \eqref{eq:os}, except that the
return $R=R_{t+1}$ can also be stochastic.\footnote{The importance of capital
    income risk for wealth dynamics is highlighted in \cite{Toda2014JET},
    \cite{benhabib2015wealth}, \cite{CaoLuo2017},
    \cite{StachurskiToda2019JET}, 
\cite{fagereng2020heterogeneity} and \cite{hubmer2020comprehensive}, among
others.} For concreteness, suppose that
\begin{equation}
    R_t=R(z_t,\xi_t)\quad \text{and}\quad y_t=y(z_t,\xi_t),\label{eq:Ry}
\end{equation} 
where $R,y$ are nonnegative measurable functions, $z_t$ is a finite state
Markov chain, and $\xi_t$ is an \iid shock that could be vector-valued.

To apply the general theory in Section~\ref{s:gf}, we assume that the utility
function $u$ is upper semicontinuous, increasing, bounded above, and
\begin{equation}
\inf_z \EE u(y(z,\xi))>-\infty.\label{eq:Ezuy3}
\end{equation}

Let us verify that the assumptions in Section~\ref{s:gf} are satisfied. The
state $x=(w,z)$ consists of the financial wealth $w$ and the exogenous Markov
state $z$. The action is consumption $a=c$. The feasible correspondence is
$\Gamma(x)=[0,w]$, which is the borrowing constraint \eqref{eq:os_borrow}. The
reward function is $r(x,a)=u(c)$. The stochastic kernel $P$ is defined through
the (exogenous) stochastic kernel of the Markov state $z$, the
distribution of the \iid shock $\xi$, and the budget constraint
\eqref{eq:os_budget}. The functions $\bar{r}$ and $\re$ in \eqref{eq:rbar} are
defined by
\begin{align*}
    \bar{r}(x)&=\sup_{0\leq c\leq w}u(c)=u(w),\\
    \re(x,a)&=\EE_{z}u(R'(w-c)+y')\\
    &\geq \EE_{z}u(y')=\EE_z u(y(z',\xi'))>-\infty,
\end{align*}
where we have used the monotonicity of $u$ and \eqref{eq:Ezuy3}. Since by
assumption $u$ is bounded above, so is $\bar{r}$. Therefore,
Assumption~\ref{a:rbar} holds. Since $u$ is upper semicontinuous, $\Gamma$ is
nonempty compact valued, and $\set{z_t}$ is a finite state Markov chain,
Assumption~\ref{a:sc} is also satisfied. Therefore, the conclusions of
Theorem~\ref{t:cs0} hold.

\begin{remark}\label{rem:MST}
    \cite{MaStachurskiToda2020JET} (henceforth MST) solve the optimal savings
    problem using the Euler equation iteration.  Our approach is different
    because it uses the (transformed) value function iteration under different
    assumptions.  While we require that the utility function is bounded above,
    MST does not require it. 
    On the other hand, MST requires the utility function to be concave, differentiable, and satisfy 
    \begin{equation}
    \sup_{z} \EE_z u'(y)<\infty.\label{eq:EzuprimeY}
    \end{equation}
    The following argument shows that our assumptions are
    weaker.\footnote{\cite{MaStachurskiToda2020JET} allow the discount factor
    $\beta$ to be random. It is straightforward to extend our theory to a
setting with stochastic discounting.} Since $u$ is concave and differentiable
under the assumptions of MST, we obtain
    \begin{equation*}
        u(1)-u(y)\leq u'(y)(1-y)\leq u'(y),
    \end{equation*}
    where we have used $u'\geq 0$ and $y\geq 0$. Taking the conditional
    expectation on $z$, we obtain
    \begin{equation*}
    u(1)-\EE_zu(y)\leq \EE_zu'(y)<\infty
    \end{equation*}
    by \eqref{eq:EzuprimeY}, implying \eqref{eq:Ezuy3}.

    More importantly, MST requires the condition $G_{\beta R}<1$, where
    $G_{\beta R}$ is the long run geometric average of $\beta R_t$. Using our
    approach, we do not require any assumption (other than nonnegativity and
    measurability) on the returns $R_t$.
\end{remark}

\subsection{Optimal Default (Continued)}
\label{ss:od_cont}

Recall the optimal default problem studied in Section~\ref{ss:are}. This
setting is a special case of our framework. In particular, 
\begin{equation*}
    x = (w, y, z), \quad
    a = (w', i), \quad
    \XX = [-b, \infty) \times \YY \times \ZZ \quad \text{and} \quad
    \AA = [-b, \infty) \times \set{0,1},
\end{equation*}
where $i$ is a discrete choice variable taking values in $\set{0,1}$, and $\YY$
and $\ZZ$ are respectively the range spaces of $\set{y_t}$ and $\set{z_t}$. The
reward function $r$ reduces to
\begin{equation*}
    r(w,y,w', i) = 
    \begin{cases*}
    u(y) & if $i=0$, \\
    u(w + y - w'/R) & if $i=1$.
    \end{cases*}
\end{equation*}
We have shown that $S \cG \subset \cG$, where $\cG$ is the set of
bounded measurable functions on $\ZZ \times [-b, \infty) \times \set{0,1}$.
Moreover, $\re$ satisfies
\begin{equation*}
    \re (z, w') 
         = \EE_{z} \max \set{ u(y'), u \left( w' + y' + b/R \right) }
        \geq \EE_z u(y'),
\end{equation*}
which is bounded below by \eqref{eq:Ezuy}. Let $\cG_1$ be the set of functions
in $\cG$ that is increasing in its second-to-last argument and upper
semicontinuous.  Through similar steps to the proof of Theorem~\ref{t:exist0},
one can show that $S \cG_1 \subset \cG_1$ and a $g$-greedy policy exists for
each $g \in \cG_1$.  As a result, all the conclusions of Theorem~\ref{t:cs0}
are true.

\subsection{Job Search (Continued)}
\label{ss:js_cont}

Recall the job search problem of Section~\ref{ss:jsp}. This problem fits into
the framework of Section~\ref{ss:t} if we let choice $a$ 
take values in $\set{0,1}$, where $0$ represents the decision to stop
and $1$ represents continue,
\begin{equation*}
    x = (w,z,c), \;\; 
    \XX = (0, \infty)^3, \;\;
    \AA = \set{0,1}, \;\;
    \Gamma(x) = \set{0,1}, \;\;
    \DD = (0, \infty)^3 \times \set{0,1}
\end{equation*}
and the reward function is
    $r(x,a) = u(w)/(1-\beta)$ if $a = 0$ and
    $r(x,a) = u(c)$ if $a=1$.
We have shown that $S \cG \subset \cG$, where $\cG$ is the set of bounded
measurable functions on $(0, \infty)$.  Note that, in this case, the function
$\re (x,a)$ reduces to $\re (z) = \EE_{z} \max \set{ u(w') / (1 - \beta),
u(c') }.$ Then $\re$ is bounded below by the inequality \eqref{eq:EmaxXY} and
\eqref{eq:Ezuy2}.  Since in addition the action set is finite, a $g$-greedy
policy always exists for each $g \in \cG$.  Let $\cG_1 \coloneqq \cG$. The
analysis above implies that all the conclusions of Theorem~\ref{t:cs0} hold.

\subsection{Optimal Savings with Endogenous Labor Choice}\label{ss:os_labor}

As another example application, consider the optimal savings problem with
endogenous labor supply
\begin{align*}
    &\maximize && \EE \sum_{t=0}^\infty \beta^tu(c_t,l_t)\\
    &\st && 0\leq c_t\leq w_t+y_tl_t,\\
    &&& 0\leq l_t\leq 1,\\
    &&& w_{t+1}=R_{t+1}(w_t-c_t+y_tl_t).
\end{align*}
Here $c_t$ is consumption, $l_t$ is labor supply, $y_t$ is wage, $w_t$ is
financial wealth at time $t$ excluding current labor income (see
Footnote~\ref{fn:timing}), and $R_{t+1}\geq 0$ is the gross return on wealth
between time $t$ and $t+1$. As before assume that $R,y$ take the form
\eqref{eq:Ry}, where $z_t$ is a finite state Markov chain and $\xi_t$ is an
\iid shock.

The state $x=(w,y,z)$ consists of the financial wealth $w$, 
wage $y$, and the exogenous Markov state $z$. The action $a=(c,l)$ consists
of consumption and labor supply. The feasible correspondence is
\begin{equation*}
    \Gamma(x)=\set{(c,l)\in \RR^2: 0\leq c\leq w+yl~\text{and}~0\leq l\leq 1}.
\end{equation*}
Suppose that the utility function $u$ is bounded above and increasing in its
first argument. The function $\bar{r}$ in \eqref{eq:rbar} is defined by
\begin{equation*}
    \bar{r}(x)=\sup_{0\leq l\leq 1}\sup_{0\leq c\leq w+yl}u(c,l)=\sup_{0\leq l\leq 1}u(w+yl,l),
\end{equation*}
which is bounded above. Noting we can bound $\bar{r}$ from below as
\begin{equation*}
    \bar{r}(x)=\sup_{0\leq l\leq 1}u(w+yl,l)\geq \sup_{0\leq l\leq 1}u(yl,l),
\end{equation*}
the function $\re$ in \eqref{eq:rbar} becomes bounded below if
\begin{equation}
    \inf_z \EE \sup_{0\leq l\leq 1}u(y(z,\xi)l,l)>-\infty,\label{eq:Ezuy4}
\end{equation}
which is analogous to \eqref{eq:Ezuy3}. In summary, by a similar argument to
Section~\ref{ss:os_cont}, the conclusions of Theorem~\ref{t:cs0} hold if $u$
is upper semicontinuous, bounded above, increasing in its first argument, and
\eqref{eq:Ezuy4} holds.

\subsection{Optimal Consumption-Portfolio Problem}\label{ss:portfolio}

As yet another example application, consider the optimal consumption-portfolio
problem
\begin{align*}
    &\maximize && \EE \sum_{t=0}^\infty \beta^t u(c_t)\\
    &\st && 0\leq c_t\leq w_t,\\
    &&& \theta_t\in \Theta(z_t),\\
    &&& w_{t+1}=R(\theta_t,z_{t+1},\xi_{t+1})(w_t-c_t)+y_{t+1}.
\end{align*}
Here $c_t$ is consumption, $z_t$ is an exogenous finite state Markov chain,
$\Theta(z_t)\subset \RR^J$ is the set of admissible portfolios of financial
assets $j=1,\dots,J$ in state $z_t$ ($\theta_t$ is a portfolio),
$y_t=y(z_t,\xi_t)$ is non-financial income ($\xi_t$ is an \iid shock that
could be vector-valued), $w_t$ is financial wealth at time $t$ including
current non-financial income, and $R(\theta_t,z_{t+1},\xi_{t+1})$ is the gross
return on wealth between time $t$ and $t+1$ given the portfolio $\theta_t$ and
shocks $(z_{t+1},\xi_{t+1})$.

This problem is a special case of our framework. The state $x=(w,z)$ consists
of the financial wealth $w$ and the exogenous Markov state $z$. The action
$a=(c,\theta)$ consists of consumption and portfolio. The feasible
correspondence is $\Gamma(x)=[0,w]\times \Theta(z)$. By the same argument as
in Section~\ref{ss:os_cont}, if the utility function $u$ is increasing,
bounded above, and \eqref{eq:Ezuy3} holds, then so does
Assumption~\ref{a:rbar}. Under additional regularity conditions ($u$ is upper
semicontinuous and the portfolio constraint $\Theta(z)$ is nonempty and
compact for each $z$), Assumption~\ref{a:sc} is satisfied and the conclusions
of Theorem~\ref{t:cs0} hold.

\section{Extensions}\label{s:ext}

In this section, we extend our theory in two important directions. First, we
illustrate how the idea of Q-transform could be extended to handle
rewards that are potentially unbounded above as well as below. Second, we
extend the theory of Section~\ref{s:gf} to solve dynamic programs with 
risk-sensitive preferences.

\subsection{Unbounded Above Rewards}\label{ss:ubabove}

In Section~\ref{s:gf}, we assume that the reward function is bounded above,
although it could be unbounded below.  To handle rewards that are potentially
unbounded above and below, we extend our theory by introducing a weighting
function $\kappa$, which is a continuous function mapping $\XX$ to $[1,
\infty)$. Let $\cG$ be the set of measurable functions $g: \DD \to \RR$ such
that $g$ is bounded below and 
\begin{equation}
\label{eq:kap_norm}
    \norm{g}_\kappa \coloneqq 
    \sup_{(x,a) \in \DD}\frac{\abs{g(x,a)}}{\kappa (x)} < \infty.
\end{equation}
The pair $(\cG, \norm{\cdot}_\kappa)$ is a Banach space (see, e.g., \cite{bertsekas2013abstract}). We make the following assumption.

\begin{assumption}\label{a:ws}
    \begin{enumerate*}
        \item\label{item:ws1} There exist constants $d \in  \RR_+$ and $\alpha
            \in (0, 1/\beta)$ such that $\bar r (x) \leq d \kappa (x)$ and
            $\EE_{x,a} \kappa (x') \leq \alpha \kappa(x)$ for all $(x,a) \in
            \DD$, and
        \item\label{item:ws2} $\re$ in \eqref{eq:rbar} is bounded below.
    \end{enumerate*}
\end{assumption}

\begin{remark}
	Note that Assumption~\ref{a:rbar} is a special case of Assumption~\ref{a:ws}
	by setting $\kappa(x)\equiv 1$ and $\alpha=1$. More importantly, 
	Assumption~\ref{a:ws} relaxes the classical weighted supremum norm 
	assumptions greatly (see, e.g., \cite{wessels1977markov} or 
	\cite{bertsekas2013abstract}), in the sense that we allow the ratio 
	of the reward function to the weighting function $\bar r/\kappa$ to 
	be unbounded from below, a case where the classical weighted supremum 
	norm approach struggles (recall Footnote~\ref{fn:ubdd_below}).
\end{remark}

Although rewards are potentially unbounded above and below, the dynamic
program can be solved by the operator $S$, as the following theorem shows.

\begin{theorem}
	\label{t:cs}
	If Assumption~\ref{a:ws} holds, then $v^*$ in \eqref{eq:szvf} is well-defined,
	\begin{enumerate}
		\item\label{item:cs.1} $S \cG \subset \cG$ and $S$ is a contraction mapping on $(\cG, \norm{\cdot}_\kappa)$,
		\item\label{item:cs.2} $S$ admits a unique fixed point $g^*$ in $\cG$, and
		\item\label{item:cs.3} $S^k g$ converges to $g^*$ at rate $O((\alpha \beta)^k)$ under $\norm{\cdot}_{\kappa}$.
	\end{enumerate}
    Moreover, if there exists a closed subset $\cG_1$ of $\cG$ such that $S \cG_1
    \subset \cG_1$ and a $g$-greedy policy exists for each $g \in \cG_1$, then
    \begin{enumerate}[label=(\alph*)]
        \item\label{item:cs.a} $g^*$ is an element of $\cG_1$ and satisfies
		\begin{equation*}
		g^*(x,a) = \beta \EE_{x,a} v^*(x') 
		\quad \text{and} \quad
		v^* (x) = \sup_{a \in \Gamma(x)} \set{ r(x,a) + g^*(x,a) },
		\end{equation*}
        \item\label{item:cs.b} at least one optimal policy exists, and
        \item\label{item:cs.c} a feasible policy is optimal if and only if it is $g^*$-greedy.
    \end{enumerate}
\end{theorem}

Let $\cG_1$ be the set of upper semicontinuous functions in $\cG$ and
\begin{equation}
\label{eq:ke}
    \ke (x,a) \coloneqq \EE_{x,a} \kappa(x').  
\end{equation}
In most applications Assumption ~\ref{a:sc} is satisfied. The following
theorem shows that the continuity of $\ke$ is sufficient for the
conclusions of Theorem~\ref{t:cs} to hold.
\begin{theorem}\label{t:exist}
	If Assumptions~\ref{a:ws} and \ref{a:sc} hold and $\ke$ in \eqref{eq:ke} is continuous, then 
    $\cG_1$ is a closed subset of $\cG$, $S \cG_1 \subset \cG_1$, and a
    $g$-greedy policy exists for each $g \in \cG_1$. Consequently, all the
    conclusions of Theorem~\ref{t:cs} hold and $g^*,v^*$ are upper
    semicontinuous.
\end{theorem}

\begin{example}
As an example application of Theorem~\ref{t:exist}, consider the optimal
savings problem \eqref{eq:os}, where the utility function $u$ can now be
unbounded both from above and below. 
Suppose that $u$ is upper semicontinuous, increasing, satisfies 
\eqref{eq:Ezuy3}, and there exist constants $p>0$ and $q\in \RR$ such that
\begin{equation}
    u(c)\leq pc+q \quad \text{for all $c>0$}.\label{eq:ucond}
\end{equation}
This condition trivially holds if $u$ is concave, and we can choose $q$
arbitrarily large.

Suppose that asset return and income take the form
\begin{equation}
    R_t=R(\xi_t)\quad \text{and}\quad y_t=y(z_t,\xi_t),\label{eq:Ryub}
\end{equation} 
where $R,y$ are nonnegative measurable functions, $z_t$ is a finite state
Markov chain, and $\xi_t$ is an \iid shock that could be
vector-valued.\footnote{Unlike the setting in Section~\ref{ss:os_cont}, the
    return is permitted to depend only on the \iid shock $\xi$.  Treating the
general case requires generalizing Theorem~\ref{t:cs} further such that
$\alpha$ depends on $x$.} In addition, assume
\begin{equation}
    Y\coloneqq \sup_{z} \EE_z y(z',\xi')<\infty, \quad \beta<1, 
    \quad \text{and} \quad \beta \EE R<1.\label{eq:Ey}
\end{equation}

As in Section~\ref{ss:os_cont}, the state is $x=(w,z)$ and the action is
$a=c$. To apply Theorem~\ref{t:exist}, define the weighting function by
$\kappa(x)=pw+q$, where $q>1$. Since
\begin{equation*}
    \bar{r}(x)=\sup_{0\leq c\leq w}u(c)=u(w)\leq pw+q=\kappa(x)
\end{equation*}
by \eqref{eq:ucond}, we can set $d=1$ in Assumption~\ref{a:ws}.  As we have
seen in Section~\ref{ss:os_cont}, the condition \eqref{eq:Ezuy3} implies that
$\re$ in \eqref{eq:rbar} is bounded below.  Therefore, to satisfy
Assumption~\ref{a:ws}, it remains to verify $\EE_{x,a}\kappa(x')\leq \alpha
\kappa(x)$ for some $\alpha\in (0,1/\beta)$. 
To this end, note that
\begin{equation*}
    \frac{\EE_{x,a}\kappa(x')}{\kappa(x)}=
    \frac{p\EE_z(R'(w-c)+y')+q}{pw+q}\leq 
    \frac{p\EE_z(R'w+y')+q}{pw+q}.
\end{equation*}
Since the right hand side is a monotone function of $w$ and achieves the
supremum at either $w=0$ or $w=\infty$, we obtain
\begin{equation}\label{eq:alpha_estim}
    \sup_{(x,a) \in \DD} 
    \frac{\EE_{x,a}\kappa(x')}{\kappa(x)}\leq \sup_z 
    \max\set{\frac{p\EE_z y'+q}{q},\EE_z R'}\leq \max\set{\frac{pY+q}{q},\EE R},
\end{equation}
where we have used the fact that $R$ does not depend on $z$ (by
\eqref{eq:Ryub}) and \eqref{eq:Ey}.  Since $q>1$ can be taken arbitrarily
large and $\frac{pY+q}{q}\to 1$ as $q\to\infty$, the right hand side of
\eqref{eq:alpha_estim} can be made arbitrarily close to $\max\set{1,\EE R}$,
which is strictly smaller than $1/\beta$ by \eqref{eq:Ey}. Therefore, we can
indeed choose $\alpha\in (0,1/\beta)$ such that 
Assumption~\ref{a:ws} holds with $d=1$ and $\kappa(x)=pw+q$ for large enough $q>1$.
Assumption~\ref{a:sc} trivially holds. Finally,
\begin{equation*}
    \ke (x,a) \coloneqq \EE_{x,a} \kappa(x')=p\EE_z(R'(w-c)+y')+q
\end{equation*}
is clearly continuous in $(x,a)=(w,z,c)$. Therefore, all the assumptions of
Theorem~\ref{t:exist} are satisfied.
\end{example}

\begin{remark}
    Under the assumption that discounting is constant and the asset return
    depends only on the \iid shock as in \eqref{eq:Ryub}, the assumptions in
    \cite{MaStachurskiToda2020JET} are strictly stronger than ours, since they
    assume \eqref{eq:Ey} and the concavity of $u$ (which implies
    \eqref{eq:ucond}). (See also Remark~\ref{rem:MST}.)
\end{remark}

\subsection{Risk-Sensitive Preferences}


We consider the general setting in Sections~\ref{s:gf} but with recursive 
(non-additive) preferences. Unlike the additively
separable case, in order to define the value function and optimality in the
recursive case, let $\gamma>0$ be the agent's risk-sensitive coefficient. 
Given any feasible policy $\sigma\in \Sigma$ and measurable function 
$v:\XX\to \RR\cup \set{-\infty}$, let
\begin{equation}
\label{eq:Tsig}
    T_\sigma v(x) \coloneqq 
    r(x,\sigma(x))-\frac{\beta}{\gamma} \log \EE_{x,\sigma(x)}\me^{-\gamma v(x')}
\end{equation}
for all $x\in \XX$ whenever the expectation is well-defined. Recall $\bar r$ 
defined in \eqref{eq:rbar}. 
In this case, the \emph{$\sigma$-value function} $v_\sigma$ is defined at 
each state $x \in \XX$ by
\begin{equation*}
    v_\sigma(x) = \limsup_{n \to \infty} T_\sigma^n \bar{r} (x).
\end{equation*}
Setting aside the issue of existence and uniqueness for now,
$v_\sigma(x)$ can be interpreted as the lifetime value of following policy
$\sigma$ forever, starting from current state $x$.  The value function $v^*$
and optimal policy $\sigma^*$ are defined as in \eqref{eq:szvf}.  The Bellman
equation associated with the dynamic program with risk-sensitive coefficient
$\gamma>0$ is
\begin{equation}\label{eq:bellmanRS}
    v(x)=\sup_{a\in \Gamma(x)}\set{r(x,a)-\frac{\beta}{\gamma} \log \EE_{x,a}\me^{-\gamma v(x')}}.
\end{equation}
Letting $g(x,a)\coloneqq -\frac{\beta}{\gamma} \log \EE_{x,a}\me^{-\gamma
v(x')}$, analogous to the derivation of \eqref{eq:tbe}, we obtain the
transformed Bellman equation
\begin{equation}\label{eq:tbeRS}
    g(x,a)=-\frac{\beta}{\gamma}\log 
    \EE_{x,a}\exp\left(-\gamma \sup_{a'\in \Gamma(x')}\set{r(x',a')+g(x',a')}\right).
\end{equation}
We define the \emph{transformed Bellman operator} $S$ by letting $Sg(x,a)$ be 
the right hand side of \eqref{eq:tbeRS} for each $g$ in the space of candidate
action-value functions (to be specified below). A feasible policy $\sigma\in\Sigma$ is
called \emph{$g$-greedy} if \eqref{eq:greedy} holds. 

We first consider the case with rewards that are bounded
above.  Let $\cG$ be the set of bounded measurable functions on $\DD$ and
$\norm{\cdot}$ be the supremum norm.   The following theorem
generalizes Theorems~\ref{t:cs0} and \ref{t:exist0} to dynamic programs with
risk-sensitive preferences.

\begin{theorem}\label{t:rs0}
    If Assumption~\ref{a:rbar} holds for $\re$ defined by
    \begin{equation}\label{eq:reRS}
    \re(x,a)\coloneqq -\frac{1}{\gamma}\log \EE_{x,a}\me^{-\gamma \bar{r}(x')},
    \end{equation}
	\begin{enumerate}
		\item\label{item:rs0.1} $S \cG \subset \cG$ and $S$ is a contraction mapping on $(\cG, \norm{\cdot})$,
		\item\label{item:rs0.2} $S$ admits a unique fixed point $g^*$ in $\cG$, and
		\item\label{item:rs0.3} $S^k g$ converges to $g^*$ at rate $O(\beta^k)$ under $\norm{\cdot}$.
	\end{enumerate}
    Moreover, if Assumption~\ref{a:sc} holds, then $v^*$ is well-defined and
    \begin{enumerate}[label=(\alph*)]
        \item\label{item:rs0.a} $g^*,v^*$ are upper semicontinuous and satisfy
		\begin{equation*}
		g^*(x,a) = -\frac{\beta}{\gamma}\log \EE_{x,a} \me^{-\gamma v^*(x')}
		\quad \text{and} \quad
		v^* (x) = \sup_{a \in \Gamma(x)} \set{ r(x,a) + g^*(x,a) },
		\end{equation*}
        \item\label{item:rs0.b} at least one optimal policy exists, and
        \item\label{item:rs0.c} a feasible policy is optimal if and only if it is $g^*$-greedy.
    \end{enumerate}
\end{theorem}

Next, we study the case with rewards that are unbounded above and below. 
In what follows, $\XX$ is a partially ordered set.
Similar to Section~\ref{ss:ubabove}, we introduce a weighting function $\kappa$,
which is continuous increasing and maps $\XX$ to $[1, \infty)$. Let 
$\cG_2$ be the set of measurable function $g :\DD \to \RR$ such that $g(x,a)$ is 
increasing in $x$, bounded below, and \eqref{eq:kap_norm} holds.

Suppose the state process evolves according to
\begin{equation}
\label{eq:x_path}
    x_{t+1} = f(x_t, a_t, \epsilon_{t+1}),
\end{equation}
where $f$ is a measurable function, and $\{\epsilon_t\}$ is an {\sc iid} 
innovation process taking values in $\RR^m$. For each $t$, we write 
$\epsilon_t = (\epsilon_{1t}, \dots, \epsilon_{mt})$ and make the 
following assumption.

\begin{assumption}
	\label{a:rs_ubdd}
	\begin{enumerate*}
		\item\label{item:rs_ubdd.1} $r(x,a)$ is increasing in $x$ and $f(x, a, \epsilon')$ 
		is increasing in $(x, \epsilon')$,
		\item\label{item:rs_ubdd.2} $\Gamma(x_1) \subset \Gamma(x_2)$ if $x_1 \leq x_2$, 
		and 
		\item\label{item:rs_ubdd.3} $\epsilon_{1t}, \dots, \epsilon_{mt}$ are independent 
		for each $t$.
	\end{enumerate*}
\end{assumption}

The following theorem extends Theorem~\ref{t:cs} and Theorem~\ref{t:exist}
to dynamic decision problems with risk-sensitive preference.

\begin{theorem}
	\label{t:rs}
	If Assumptions~\ref{a:ws} and \ref{a:rs_ubdd} hold for $\hat r$ defined in
	\eqref{eq:reRS}, then 
	\begin{enumerate}
		\item\label{item:rs.1} $S \cG_2 \subset \cG_2$ and $S$ is a contraction 
		mapping on $(\cG_2, \|\cdot\|_\kappa)$.
		\item\label{item:rs.2} $S$ admits a unique fixed point $g^*$ in $\cG_2$, and
		\item\label{item:rs.3} $S^kg$ converges to $g^*$ at rate $O((\alpha \beta)^k)$ 
		under $\|\cdot\|_\kappa$.
	\end{enumerate}
	Moreover, if Assumption~\ref{a:sc} holds, then $v^*$ is well-defined and 
	\begin{enumerate}[label=(\alph*)]
		\item\label{item:rs.a} $g^*, v^*$ are upper semicontinuous and satisfy
		\begin{equation*}
			g^*(x,a) = - \frac{\beta}{\gamma} \log \EE_{x,a} \me^{-\gamma v^*(x')}
			\quad \text{and} \quad
			v^*(x) = \sup_{a \in \Gamma(x)} \left\{
			    r(x,a) + g^*(x,a)
			\right\},
		\end{equation*}
		\item\label{item:rs.b} at least one optimal policy exists, and 
		\item\label{item:rs.c} a feasible policy is optimal if and only if it is 
		$g^*$-greedy.		
	\end{enumerate}
\end{theorem}

\begin{example}
	\label{ex:opt_grw}
	\cite{bauerle2018stochastic} study an optimal growth model with risk-sensitive 
	preference. In their setting,
	\begin{equation*}
	    \XX = \AA = \RR_+, \quad
	    \Gamma(x) = [0, x] \quad \text{and} \quad
	    r(x,a) = u(x-a),
	\end{equation*}
	where $x$ is capital, $a$ is the amount of investment, and $u$ is the utility function.
	\cite{bauerle2018stochastic} assume that the utility function $u$ is bounded below. 
		
	Consider, for example, $u(c) = \log c$ and
	\begin{equation*}
	    x_{t+1} = f(a_t, \epsilon_{t+1}) = \eta a_t + \epsilon_{t+1}, 
	    \quad \{\epsilon_t\} \iidsim LN(\mu, \sigma^2),
	\end{equation*}
	where $\eta > 0$. This setup is common in applied works and is not covered 
	by \cite{bauerle2018stochastic}, because the logarithmic utility is unbounded 
	below. However, Theorem~\ref{t:rs} can be applied. Clearly, 
	Assumptions~\ref{a:sc} and \ref{a:rs_ubdd} hold. Moreover, since
	$\bar r(x) = u(x)$ on $\XX$, for all $(x,a) \in \DD$,
	\begin{equation*}
	    \EE _{x,a} \me^{-\gamma \bar r(x')} 
	    =\EE \me^{-\gamma u(\eta a + \epsilon')} 
	    \leq \EE \me^{-\gamma u(\epsilon')} 
	    = \me^{-\gamma \mu + \gamma^2 \sigma^2/2} < \infty.
	\end{equation*}
	Hence $\hat r(x,a) \geq \mu - \gamma \sigma^2/2$
	on $\DD$ and Assumption~\ref{a:ws}\ref{item:ws2} holds.
	Let $\bar \epsilon \coloneqq \EE \epsilon_t$.
	\begin{itemize}
		\item If $\eta \leq 1$, then Assumption~\ref{a:ws}\ref{item:ws1} holds
		for all $\beta \in (0,1)$ by letting $\alpha \in (1, 1/\beta)$
		and $\kappa(x) \coloneqq x + \bar \epsilon /(\alpha -1)$.
		\item If $\eta > 1$, then Assumption~\ref{a:ws}\ref{item:ws1} holds
		for all $\beta \in (0,1/\eta)$ by letting $\alpha \coloneqq \eta$
		and $\kappa(x) \coloneqq x + \bar \epsilon / (\alpha - 1)$.
	\end{itemize}
    Assumption~\ref{a:ws} is now verified. Therefore, all the conclusions
    of Theorem~\ref{t:rs} hold.
\end{example}

\begin{remark}
	In the state evolution path \eqref{eq:x_path}, we impose $\{\epsilon_t\}$ to be an 
	{\sc iid} innovation process. Indeed, this restriction can be relaxed. For 
	example, we can set $\{\epsilon_t\}$ to be a finite Markov chain, or 
	$\log \epsilon_t = \mu(z_t) + \sigma(z_t) \xi_t$, where $\{z_t\}$ is a finite
	Markov chain and $\{\xi_t\}$ is an {\sc iid} innovation process with finite
	moment generating function, etc.
	By expanding the state vector to accommodate the exogenous state $\epsilon_t$ 
	or $z_t$ and adjusting Assumption~\ref{a:rs_ubdd} mildly, the conclusions of
	Theorem~\ref{t:rs} still hold in these generalized settings.
\end{remark}

\section{Conclusion}
\label{s:conc}

We proposed a new approach to solving dynamic programs with unbounded rewards,
based on Q-transforms. The essence of our approach lies in
transforming an originally unbounded dynamic program into a bounded one. We
demonstrated via a range of applications that our method fits well with
stochastic dynamic decision problems with unbounded below rewards and
can be extended to handle crucial dynamic programs that are both
unbounded above and unbounded below. In particular, the Q-transform approach
is not limited to solving dynamic programs that are additively separable. We
showed that dynamic decision problems with risk-sensitive preference and
unbounded (above and below) period reward functions are also covered by
Q-transform. Although exploring further recursive preference decision problems
via our approach goes beyond the scope of this study, the theory of
Q-transform presented here should serve as a solid foundation for new work
along these lines.

\section{Appendix: Proof of Main Results}
\label{s:appendix}

Since Theorems~\ref{t:cs0}, \ref{t:exist0} are special cases of
Theorems~\ref{t:cs}, \ref{t:exist} by setting $\kappa(x)\equiv 1$ and
$\alpha=1$, we only prove the latter.
We first show that the $\sigma$-value function $v_\sigma$ in \eqref{eq:vsigma}
and the value function $v^*$ in \eqref{eq:szvf} are well-defined.

\begin{lemma}\label{lm:vsigma}
    If Assumption~\ref{a:ws}\ref{item:ws1} holds, then for any feasible policy
    $\sigma\in \Sigma$ and initial state $x_0=x\in \XX$, the quantities
    $v_\sigma(x)$ in \eqref{eq:vsigma} and $v^*(x)$ in \eqref{eq:szvf} are
    well-defined in $\RR \cup \set{-\infty}$.
\end{lemma}

\begin{proof}
    Using the definition of $\bar{r}$ in \eqref{eq:rbar} and
    Assumption~\ref{a:ws}\ref{item:ws1}, the $t$-th term on the right hand
    side of \eqref{eq:vsigma} can be bounded above as
    \begin{equation*}
        \beta^t \EE_x r(x_t,\sigma(x_t))\leq \beta^t \EE_x
        \bar{r}(x_t)\leq \beta^t \EE_x d\kappa(x_t)\leq
         \beta^t\alpha^t d\kappa(x)=d\kappa(x)(\alpha\beta)^t.
    \end{equation*}
    Since by assumption $0<\alpha\beta<1$, summing over $t$, we obtain
    \begin{equation*}
        v_\sigma(x)
        =\sum_{t=0}^\infty \beta^t \EE_x r(x_t,\sigma(x_t))
        \leq \sum_{t=0}^\infty d\kappa(x)(\alpha\beta)^t=\frac{d\kappa(x)}{1-\alpha\beta}<\infty.
    \end{equation*}
    Therefore, $v_\sigma(x)$ in \eqref{eq:vsigma} is well-defined in $\RR \cup
    \set{-\infty}$. Taking the supremum over $\sigma\in \Sigma$, we obtain
    \begin{equation*}
        v^*(x)=\sup_{\sigma\in\Sigma} 
        v_\sigma(x)\in \left[-\infty,\frac{d\kappa(x)}{1-\alpha\beta}\right],
    \end{equation*}
    so $v^*(x)$ in \eqref{eq:szvf} is also well-defined.
\end{proof}

\begin{proof}[Proof of Theorem~\ref{t:cs}]
    To see claim~\ref{item:cs.1} holds, we first show that $S \cG \subset
    \cG$. Fix $g \in \cG$. By the definition of $\cG$, there is a lower bound
    $L \in \RR$ such that $g \geq L$. Then
    \begin{align*}
        S g (x,a) 
        &\geq \beta \EE_{x,a} \sup_{a' \in \Gamma(x')} \set{
        r(x',a') + L } = \beta \EE_{x,a} \left[
            \sup_{a' \in \Gamma(x')} r(x',a') + L 
        \right]  \\
        &= \beta \left[\EE_{x,a} \bar{r} (x') + L \right] 
        = \beta \left[ \re (x,a) + L \right].
    \end{align*}
    Since by assumption $\re$ is bounded below, so is $Sg$. Moreover, by Assumption~\ref{a:ws},
    \begin{align*}
        S g (x,a) &\leq \beta \EE_{x,a} \set{ \bar r (x')
        + \sup_{a' \in  \Gamma(x')} g(x',a') }    \\
        &\leq \beta \EE_{x,a}(d +  \norm{g}_\kappa) \kappa (x') 
        \leq \alpha \beta (d + \norm{g}_\kappa) \kappa (x) 
    \end{align*}
    for all $(x,a ) \in \DD$. Hence, $Sg / \kappa$ is bounded above. Since in
    addition $Sg$ is bounded below and $\kappa \geq 1$, we have
    $\norm{Sg}_\kappa < \infty$, implying $Sg \in \cG$.
    
    Obviously, $S$ is an monotone operator, i.e., $S g_1 \leq S g_2$ whenever
    $g_1 \leq g_2$. To see that $S$ is a contraction mapping on $(\cG,
    \norm{\cdot}_\kappa)$ of modulus $\alpha \beta$, it suffices to show
    that\footnote{For all $g_1, g_2 \in \cG$, we have 
    	$g_1(x,a) \leq g_2(x,a) + \norm{g_1-g_2}_\kappa \kappa (x)$. 
    	The monotonicity of $S$ and \eqref{eq:g_disc} then imply that 
    	$Sg_1(x,a) \leq Sg_2(x,a) + \alpha \beta \norm{g_1-g_2}_\kappa \kappa(x)$. 
    	Switching the roles of $g_1$ and $g_2$ yields 
    	$\norm{Sg_1 - Sg_2}_\kappa \leq \alpha \beta \norm{g_1 - g_2}_\kappa$.}
    \begin{equation}
    \label{eq:g_disc}
        S(g+K\kappa)(x,a) \leq Sg(x,a) + \alpha \beta K \kappa(x)
        \quad \text{for all } K \in \RR_+.
    \end{equation}
    Condition~\eqref{eq:g_disc} obviously holds, because by Assumption~\ref{a:ws}, we have
    \begin{align*}
        S(g+M\kappa)(x,a)
        &= \beta \EE_{x,a} \sup_{a' \in  \Gamma(x')} 
            \set{ r(x',a') + g_1 (x',a') + K \kappa (x')}   \\
        &= \beta \EE_{x,a} \sup_{a' \in  \Gamma(x')} 
            \set{ r(x',a') + g_1 (x',a')} + 
            \beta K \EE_{x,a} \kappa (x') \\
        &\leq Sg (x,a) + \alpha \beta K \kappa(x).
    \end{align*}
    Hence $S$ is a contraction mapping on $(\cG, \norm{\cdot}_\kappa)$ and
    claim~\ref{item:cs.1} is verified. Claims~\ref{item:cs.2} and
    \ref{item:cs.3} follow immediately from claim~\ref{item:cs.1} and the
    Banach contraction mapping theorem. 

    To see that claims~\ref{item:cs.a}--\ref{item:cs.c} hold, let $\cV$ be 
    the set of measurable functions $v: \XX \to \RR \cup \set{-\infty}$ 
    such that $(x,a) \mapsto \beta \EE_{x,a} v(x')$ is in $\cG$, and 
    define the operators $W$ on $\cV$ and $M$ on $\cG$ respectively as
    \begin{equation*}
        W v (x,a) \coloneqq \beta \EE_{x,a} v(x')
        \quad \text{and} \quad
        M g (x) \coloneqq \sup_{a \in \Gamma(x)} \left\{
            r(x,a) + g(x,a)
        \right\}.
    \end{equation*}
    Then the original Bellman operator $T$ (i.e., $Tv$ equals the right hand 
    side of \eqref{eq:bellman} given $v \in \cV$) and the transformed 
    Bellman operator $S$ in \eqref{eq:tbe} can be written as
    \begin{equation*}
        T = M W \quad \text{and} \quad S = W M.
    \end{equation*}
    In particular, for each $g \in \cG$, because $S$ maps $\cG$ into itself 
    as was shown, $W (Mg)= W M g = Sg \in \cG$. Hence, $Mg \in \cV$ by the 
    definition of $\cV$. As $g$ is chosen arbitrarily, this implies that $M$ 
    maps $\cG$ into $\cV$, and thus $T$ maps $\cV$ into itself.
    
	Since $\cG_1$ is a closed subset of $\cG$ and $S \cG_1 \subset \cG_1$, $S$ 
	is also a contraction mapping on $(\cG_1, \norm{\cdot}_\kappa)$ and the 
	unique fixed point $g^*$ of $S$ indeed lies in $\cG_1$. Next, we show 
	that $\bar v \coloneqq M g^*$ is a fixed point of $T$ in $\cV$. 
	Note that $\bar v \in \cV$ because $M \cG \subset \cV$. Moreover,
	\begin{equation*}
	    T \bar v = M W \bar v = M W M g^* 
	    = M Sg^* = M g^* = \bar v. 
	\end{equation*}
	Therefore, $\bar v$ is a fixed point of $T$ in $\cV$, as was to be shown. 
	
    
    Since $\bar v = M g^*$ and $g^*=Sg^*$ as were shown, we have 
    $g^* = WM g^* = W \bar v$. To verify claim~\ref{item:cs.a}, it remains to show that $\bar{v}$ equals
    the value function $v^*$ in \eqref{eq:szvf}. For all $x_0 \in \XX$ and
    $\sigma \in \Sigma$, because $\bar v = T \bar v$, the definition of $T$ 
    implies that
    \begin{align}\label{eq:vbar}
    	\bar v(x_0) &\geq r(x_0, \sigma(x_0)) + \beta \EE_{x_0,\sigma(x_0)} \bar{v} (x_1)  \nonumber  \\
    	& \geq r(x_0, \sigma(x_0)) + \beta \EE_{x_0,\sigma(x_0)} 
    	\set{ r(x_1, \sigma(x_1)) + \beta \EE_{x_1,\sigma(x_1)} \bar{v} (x_2) }  \nonumber  \\
    	&= r(x_0, \sigma(x_0)) + \beta \EE_{x_0,\sigma(x_0)} 
    	r(x_1, \sigma(x_1)) + \beta^2 \EE_{x_0,\sigma(x_0)} \EE_{x_1,\sigma(x_1)} \bar{v} (x_2)  \nonumber  \\
    	&\geq \sum_{t=0}^N \beta^t \EE_{x_0,\sigma(x_0)} \cdots \EE_{x_{t-1},\sigma(x_{t-1})} r(x_t, \sigma(x_t))  
    	 + \beta^{N+1} \EE_{x_0,\sigma(x_0)} \cdots \EE_{x_{N},\sigma(x_{N})} \bar v (x_{N+1})  \nonumber  \\
    	&= \sum_{t=0}^N \beta^t \EE_{x_0} r(x_t, \sigma(x_t)) + 
    	\beta^{N} \EE_{x_0,\sigma(x_0)} \cdots \EE_{x_{N-1},\sigma(x_{N-1})} g^* (x_N, \sigma(x_N)).
    \end{align}
    Notice that, by Assumption~\ref{a:ws}\ref{item:ws1}, we have
    \begin{align*}
        & \abs{\beta^{N} \EE_{x_0,\sigma(x_0)} \cdots \EE_{x_{N-1},\sigma(x_{N-1})} g^* (x_N, \sigma(x_N)) }    \\
        & \leq \beta^{N} \EE_{x_0,\sigma(x_0)} \cdots \EE_{x_{N-1},\sigma(x_{N-1})} \abs{ g^* (x_N, \sigma(x_N)) }    \\
        & \leq \beta^{N} \EE_{x_0,\sigma(x_0)} \cdots \EE_{x_{N-1},\sigma(x_{N-1})} \norm{g^*}_\kappa \kappa (x_N)   \\
        & \leq (\alpha \beta)^N \norm{g^*}_\kappa \kappa(x_0) \to 0
        \quad \text{as}~N \to \infty.
    \end{align*}
    Letting $N \to \infty$ in \eqref{eq:vbar}, Lemma~\ref{lm:vsigma} implies 
    that $\bar v (x_0) \geq v_\sigma(x_0)$. Since $x_0 \in \XX$ and 
    $\sigma \in \Sigma$ are
    arbitrary, we have $\bar v \geq v^*$. Moreover, because $g^* = W \bar v$
    and $g^* \in \cG_1$ implies that a $g^*$-greedy policy $\sigma^*$ exists, all the
    inequalities in \eqref{eq:vbar} hold with equality once we let $\sigma =
    \sigma^*$. In other words, we have $\bar{v} = v_{\sigma^*} \leq v^*$. In
    summary, we have shown that $\bar v = v^*$. Hence, $g^* = W v^*$,
    $v^* = M g^*$, and claim~\ref{item:cs.a} holds.
    
    Moreover, the above arguments also imply that $\sigma^*$ is an optimal 
    policy (i.e., $v^* = v_{\sigma^*}$) if it is $g^*$-greedy. 
    Because a $g^*$-greedy policy exists by assumption,
    the set of optimal policies is nonempty and claim~\ref{item:cs.b} holds.
    To see that claim~\ref{item:cs.c} holds, it remains to show that  
    any optimal policy $\sigma^*$ is $g^*$-greedy. Note that 
    \begin{align*}
        r(x, \sigma^*(x)) + \beta g^*(x, \sigma^*(x))
        &= r(x, \sigma^*(x)) + \beta \EE_{x,\sigma^*(x)} v^* (x')  \\
        &= r(x, \sigma^*(x)) + \beta \EE_{x,\sigma^*(x)} v_{\sigma^*} (x') \\
        &= v_{\sigma^*} (x) = v^*(x) = T v^* (x) = M W v^* (x) = M g^*(x),
    \end{align*}
    where the first and last equalities hold since $g^* = W v^*$, the second 
    and fourth equalities hold because $\sigma^*$ is optimal, the 
    third equality equality holds by the definition of $v_{\sigma^*}$,
    and the fifth equality holds because $v^*$ is a fixed point of $T$ as 
    was shown above. Therefore, $\sigma^*$ is a $g^*$-greedy policy. 
    We have now shown that claim~\ref{item:cs.c} holds.
\end{proof}

\begin{proof}[Proof of Theorem~\ref{t:exist}]
	
    To apply Theorem~\ref{t:cs}, it suffices to prove that $\cG_1$ is a closed
    subset of $\cG$, $S \cG_1 \subset \cG_1$, and that a $g$-greedy policy exists
    for each $g \in \cG_1$. 
	
    To show that $\cG_1$ is a closed subset, let $\set{g_n}$ be a sequence in
    $\cG_1$ such that $\norm{g_n - g_0}_\kappa \to 0$ for some $g_0\in \cG$.
    Because $(\cG, \norm{\cdot}_\kappa)$ is complete, it suffices to show that
    $g_0$ is upper semicontinuous. For all $(x_0, a_0) \in \DD$ and $y >
    g_0(x_0,a_0)$. Let $\epsilon \coloneqq y - g_0(x_0, a_0)$. Since $\kappa$
    is continuous and $g_n$ is upper semicontinuous for all $n$, there exist
    $N \in \NN$ and a neighborhood $B$ of $(x_0, a_0)$ such that for all
    $(x,a) \in B$,
	\begin{equation*}
        \abs{g_N (x,a) - g_0(x, a)} < \epsilon/3
		\quad \text{and} \quad
		g_N(x, a) < g_N(x_0, a_0) + \epsilon/3.
 	\end{equation*}
    Hence, $g_0(x,a) < g_N(x,a) + \epsilon/3 < g_N (x_0,a_0) + 2\epsilon/3 <
    g_0(x_0, a_0) + \epsilon = y$ for each $(x,a) \in B$, implying that $g_0$
    is upper semicontinuous. 
	
    Fix $g \in \cG_1$. Note that $r + g$ is upper semicontinuous because both
    $r$ and $g$ are upper semicontinuous. Since, in addition, $\Gamma$ is
    compact-valued and upper hemicontinuous, Lemma~1 of
    \cite{jaskiewicz2011discounted} implies that a $g$-greedy policy exists,
    and that $x \mapsto h_g (x) \coloneqq \sup_{a \in \Gamma(x)} \set{ r(x,a) +
    g(x,a)}$ is upper semicontinuous.
	
    Since $Sg \in \cG$ by Theorem~\ref{t:cs}, to see that $S \cG_1 \subset \cG_1$,
    it remains to show that $Sg$ is upper semicontinuous.
    Assumption~\ref{a:ws} and the definition of $\cG_1$ yield $h_g \leq (d +
    \norm{g}_\kappa) \kappa$, so Lemma~7 of \cite{jaskiewicz2011discounted}
    implies $Sg (x,a) = \beta \EE_{x,a} h_g (x')$ is upper semicontinuous.
    Hence, $S\cG_1 \subset \cG_1$.
	
	All the claims of Theorem~\ref{t:exist} then follow from Theorem~\ref{t:cs}.
\end{proof}

\begin{proof}[Proof of Theorem~\ref{t:rs0}]
    To see claims~\ref{item:rs0.1}--\ref{item:rs0.3} hold, we first show that
    $S \cG \subset \cG$. Fix $g \in \cG$. Since $g\geq -\norm{g}$, we obtain
    \begin{align*}
        S g (x,a) &\geq -\frac{\beta}{\gamma}\log \EE_{x,a}\exp\left(
            -\gamma \sup_{a'\in \Gamma(x')}\set{r(x',a')-\norm{g}}
        \right)\\
		&=-\frac{\beta}{\gamma}\log \EE_{x,a}\me^{-\gamma (\bar{r}(x')-\norm{g})}=\beta[\re (x,a) -\norm{g}],
    \end{align*}
    where the last equality uses \eqref{eq:reRS}.  Since by assumption $\re$
    is bounded below, so is $Sg$.  A similar argument yields $	Sg(x,a)\leq
    \beta[\re (x,a)+\norm{g}]$, so $Sg$ is bounded above. This shows that $S$
    is a self map on $\cG$.

    To show that $S$ is a contraction mapping, we verify
    \cite{blackwell1965discounted}'s sufficient conditions. $S$ is clearly
    monotone. Let $h(x,a) \coloneqq r(x,a) + g(x,a)$. If $K \geq 0$ is any 
    constant, then for any $g\in \cG$ we have
    \begin{align*}
        S(g+K)(x,a)&=-\frac{\beta}{\gamma}\log 
            \EE_{x,a}\exp\left(
                -\gamma \sup_{a'\in \Gamma(x')}\set{h(x',a')+K}
            \right)  \\
        &= -\frac{\beta}{\gamma}\log 
            \EE_{x,a}\exp\left(
                -\gamma \sup_{a'\in \Gamma(x')}h(x',a')
            \right) + \beta K
        = Sg(x,a) + \beta K,
    \end{align*}
    so the discounting property holds. Therefore,
    claims~\ref{item:rs0.1}--\ref{item:rs0.3} hold.
    
    Let $\cV$ be all measurable maps $v: \XX \to \RR \cup \{-\infty\}$
    such that $(x,a) \mapsto -\frac{\beta}{\gamma} \log \EE_{x,a} \me^{-\gamma v(x')}$
    is in $\cG$. Define the operators $W$ on $\cV$ and $M$ on $\cG$ 
    respectively as 
    \begin{equation}
    \label{eq:wm}
	    W v(x,a) \coloneqq -\frac{\beta}{\gamma} \log \EE_{x,a} \me^{-\gamma v(x')}
	    \quad \text{and} \quad
	    M g(x) \coloneqq \sup_{a \in \Gamma (x)} \left\{
	    r(x,a) + g(x,a)
	    \right\}.
    \end{equation}
    The original Bellman operator $T$ and the transformed Bellman operator
    $S$ can then be written as $T = MW$ and $S = W M$. 
    Let $\cG_1$ be the set of upper semicontinuous functions in $\cG$. 
    Similar to the proof of Theorem~\ref{t:cs}, $\bar v \coloneqq M g^*$ 
    is a fixed point of $T$ in $\cV$ and $g^* = W \bar v$. Moreover,
    a similar argument to the proof of Theorem~\ref{t:exist} shows that
    $\cG_1$ is a closed subset of $\cG$, $S$ maps $\cG_1$ into itself, 
    the unique fixed point $g^*$ of $S$ is in $\cG_1$, and 
    a $g$-greedy policy exists for each $g \in \cG_1$. 
    To see that claim~\ref{item:rs0.a} holds, it remains to verify $v^* = \bar v$.
    
    Fix $\sigma \in \Sigma$. The definition of $T$ and the monotonicity
    of $T_\sigma$ in \eqref{eq:Tsig} imply that
    \begin{equation}
        \label{eq:vbar_ineq}
        \bar v = T \bar v \geq T_\sigma \bar v \geq \dots \geq T_\sigma^n \bar v
    \end{equation}
    for all $n \in \NN$. One can show that, for all constant
    $\ell \in \RR$ and $x \in \XX$, 
    \begin{equation}
    \label{eq:Tnsig}
        T_\sigma^n(\bar r + \ell) (x) = T_\sigma^n \bar r(x) + \beta^n \ell.
    \end{equation}
    Since $\beta \in (0,1)$, letting $n \to \infty$ yields 
    \begin{equation*}
        v_\sigma (x) = \limsup_{n \to \infty} T_\sigma^n \bar r(x)
        = \limsup_{n \to \infty} T_\sigma^n (\bar r + \ell) (x)
    \end{equation*}
    for all $x \in \XX$ and $\ell \in \RR$. Since
    $\bar v \in [\bar r - \|g^*\|, \bar r + \|g^*\|]$, by the monotonicity
    of $T_\sigma$, 
    \begin{equation}
    \label{eq:Tnvbar}
        \limsup_{n \to \infty} T_\sigma^n \bar v (x) 
        =\limsup_{n \to \infty} T_\sigma^n \bar r (x) = v_\sigma(x)
    \end{equation}
    for all $x \in \XX$.
    Letting $n\to \infty$ in \eqref{eq:vbar_ineq} then yields $\bar v \geq v_\sigma$.
    Since $\sigma$ is chosen arbitrarily, this implies $\bar v \geq v^*$. To see 
    conversely $\bar v \leq v^*$, we define $M_\sigma$ at each $g \in \cG$ by
    \begin{equation}
    \label{eq:Msig}
    	M_\sigma g(x) \coloneqq r(x, \sigma(x)) + g(x, \sigma(x)).
    \end{equation}
    Then $T_\sigma = M_\sigma W$. Since $\bar v$ is a fixed point of $T$, 
    $g^* = W \bar v$, and a $g^*$-greedy policy $\sigma^*$ exists as were shown, 
    we have
    \begin{equation*}
    	\bar v = T \bar v = MW \bar v = M g^* = M_{\sigma^*} g^* 
    	= M_{\sigma^*} W \bar v = T_{\sigma^*} \bar v.
    \end{equation*}
    Hence, $\bar v = T_{\sigma^*}^n \bar v$ for all $n \in \NN$. Letting
    $n \to \infty$ and then using \eqref{eq:Tnvbar} yield 
    $\bar v = v_{\sigma^*} \leq v^*$. In summary, we have 
    shown that $\bar v = v^*$. Claim~\ref{item:rs0.a} is now verified.
    
    The above arguments also imply that a policy is optimal if it is $g^*$-greedy,
    and an optimal policy exits (since there exists a $g^*$-greedy policy). Hence
    claim~\ref{item:rs0.b} holds. To see that claim~\ref{item:rs0.c} holds, it remains
    to show that any optimal policy $\sigma^*$ must be $g^*$-greedy, equivalently,
    $v^* = v_{\sigma^*}$ implies $M_{\sigma^*} g^* = M g^*$.
    
    To see this, because $\bar v = v^*$ and $g^* = W \bar v$, we have
    \begin{equation*}
        M_{\sigma^*} g^* = M_{\sigma^*} W v^* = T_{\sigma^*} v^*
        \quad \text{and} \quad 
        Mg^* = MWv^* = Tv^*. 
    \end{equation*}
    Since $v^* = v_{\sigma^*}$ and 
    $v_{\sigma^*} = \limsup_{n \to \infty} T_{\sigma^*}^n v^*$ as shown above, 
    \begin{equation}
    \label{eq:Ms*}
    	M_{\sigma^*} g^* = 
    	T_{\sigma^*} \left(
    	    \limsup_{n \to \infty} T_{\sigma^*}^n v^*  
    	\right).
    \end{equation}
    Because $\me^{-\gamma \limsup_{n \to \infty}x_n} = 
    \liminf_{n \to \infty} \me^{-\gamma x_n}$ for a given sequence $\{x_n\}$, 
    the Fatou's lemma implies that, for all $x \in \XX$,
    \begin{equation*}
        \EE_{x, \sigma^*(x)} 
        \me^{-\gamma \limsup_{n \to \infty} T_{\sigma^*}^n v^*(x')}
        \leq \liminf_{n \to \infty}
        \EE_{x, \sigma^*(x)} 
        \me^{-\gamma T_{\sigma^*}^n v^*(x')}.
    \end{equation*}
    So $W (\limsup_{n \to \infty} T_{\sigma^*}^n v^*) \geq 
    \limsup_{n \to \infty} W(T_{\sigma^*}^n v^*)$. Applying $M_\sigma^*$ on both 
    sides and then using \eqref{eq:Ms*}, $T_{\sigma^*} = M_{\sigma^*} W$, and the 
    definition of $M_{\sigma^*}$ yield
    \begin{equation*}
        M_{\sigma^*} g^* = T_{\sigma^*} \left(
            \limsup_{n \to \infty} T_{\sigma^*}^n v^*  
        \right)
        \geq \limsup_{n \to \infty} T_{\sigma^*}^{n+1} v^*
        = v_{\sigma^*}
        = v^* = Tv^* = Mg^*.
    \end{equation*}
    Because $M_{\sigma^*} g^* \leq M g^*$ by definition, we must have 
    $M_{\sigma^*} g^* = M g^*$, equivalently, $\sigma^*$ is $g^*$-greedy.
    Claim~\ref{item:rs0.c} is verified and the proof is now complete. 
\end{proof}

\begin{proof}[Proof of Theorem~\ref{t:rs}]
	Let $\cV_2$ be the set of measurable maps $v: \XX \to \RR \cup \{-\infty\}$
	such that $(x,a) \mapsto -\frac{\beta}{\gamma} \log \EE_{x,a} \me^{-\gamma v(x')}$
	is in $\cG_2$. Define the operators $W$ on $\cV_2$ and $M$ on $\cG_2$ 
	as in \eqref{eq:wm}. By Jensen's inequality and 
	Assumption~\ref{a:ws}\ref{item:ws1}, for all $(x,a) \in \DD$ and $K \in \RR_+$, 
	we have
	\begin{equation}
	\label{eq:W_upbd}
		W(K \kappa) (x,a) = -\frac{\beta}{\gamma} \log \EE_{x,a}
		    \me^{-\gamma K \kappa(x')}
		\leq \beta \EE_{x,a} K \kappa (x') 
		\leq \alpha \beta K \kappa(x). 
	\end{equation}
	We first show that $S \cG_2 \subset \cG_2$. Fix $g \in \cG_2$. 
	Assumption~\ref{a:ws}\ref{item:ws1} implies that
	\begin{equation*}
	    Mg(x) \leq \bar r (x) + \|g\|_\kappa \kappa (x) 
	    \leq (d+ \|g\|_\kappa) \kappa (x)
	\end{equation*}
	for all $x \in \XX$. Using the monotonicity of $W$ and \eqref{eq:W_upbd} 
	then gives
	\begin{equation*}
	    Sg(x,a) = WMg(x,a)
        \leq W ((d + \|g\|_\kappa) \kappa) (x,a) 
        \leq \alpha \beta (d + \|g\|_\kappa) \kappa(x)
	\end{equation*}
	for all $(x,a) \in \DD$. Hence, $Sg/\kappa$ is bounded above. Moreover,
	a similar argument to the proof of Theorem~\ref{t:rs0} shows that $Sg$ 
	is bounded below.
	
	To see that $Sg (x,a)$ is increasing in $x$, let $x_1, x_2 \in \XX$ with 
	$x_1 \leq x_2$ and fix $a \in \Gamma(x_1)$. By 
	Assumption~\ref{a:rs_ubdd}\ref{item:rs_ubdd.1}--\ref{item:rs_ubdd.2}, 
	$x_1' := f(x_1, a, \epsilon') \leq f(x_2,a,\epsilon') =: x_2'$ and 
	$\Gamma(x_1') \subset \Gamma(x_2')$. Since in addition $r(x,a)$ and $g(x,a)$ 
	are increasing in $x$, we have
	\begin{align*}
	    Mg(x_1') &= \sup_{a' \in \Gamma(x_1')} \left\{
	        r(x_1',a') + g(x_1',a')
	    \right\}
	    \leq \sup_{a' \in \Gamma(x_2')} \left\{
	    r(x_1',a') + g(x_1',a')
	    \right\}  \\ 
	    &\leq \sup_{a' \in \Gamma(x_2')} \left\{
	    r(x_2',a') + g(x_2',a')
	    \right\} = Mg(x_2').
	\end{align*}
	Noting that $Mg(x_i') = Mg(f(x_i ,a, \epsilon'))$ for $i = 1,2$, 
	\begin{align*}
	    Sg(x_1, a)  
	    = -\frac{\beta}{\gamma} \log \EE \me^{-\gamma Mg(f(x_1 ,a, \epsilon'))}  
	    \leq -\frac{\beta}{\gamma} \log \EE \me^{-\gamma Mg(f(x_2 ,a, \epsilon'))} 
	    = Sg(x_2, a).
	\end{align*}
	Hence, $Sg(x,a)$ is increasing in $x$. We have now shown that 
	$S \cG_2 \subset \cG_2$. 
	
	Next, we show that $S$ is a contraction on $(\cG_2, \|\cdot\|_\kappa)$.
	Let $h_1, h_2 : \XX \to \RR$ be increasing functions 
	on $\XX$. By Assumption~\ref{a:rs_ubdd}\ref{item:rs_ubdd.1}, 
	$\epsilon' \mapsto h_i(f(x,a,\epsilon'))$ is increasing for all $(x,a) \in \DD$ and 
	$i = 1,2$. Since Assumption~\ref{a:rs_ubdd}\ref{item:rs_ubdd.3} implies that 
	$\epsilon'$ is independent across dimensions, applying the 
	Fortuin–Kasteleyn–Ginibre inequality \citep{fortuin1971correlation} gives
	\begin{align*}
		\EE_{x,a} \me^{-\gamma (h_1 + h_2) (x')} 
		&= \EE \me^{-\gamma h_1(f(x,a,\epsilon'))} 
		    \me^{-\gamma h_2(f(x,a,\epsilon'))}  \\
		&\geq \EE \me^{-\gamma h_1(f(x,a,\epsilon'))}
		    \EE \me^{-\gamma h_2(f(x,a,\epsilon'))}
		= \EE_{x,a} \me^{-\gamma h_1 (x')}
		    \EE_{x,a} \me^{-\gamma h_2 (x')}
	\end{align*}
	for all $(x, a) \in \DD$ whenever the expectation is well-defined.
	Therefore,
	\begin{equation}
	\label{eq:fkg}
	    W(h_1 + h_2)(x,a) \leq W h_1(x,a) + Wh_2(x,a)
	\end{equation}
	for all $(x,a) \in \DD$ and real-valued increasing functions $h_1, h_2$ on $\XX$.
	By the definition of $M$, for all $x \in \XX$,
	\begin{equation*}
	    M (g+ K \kappa) (x) = Mg(x) + K \kappa(x).
	\end{equation*}
	Because both $Mg$ and $\kappa$ are increasing on $\XX$, applying \eqref{eq:W_upbd} 
	and \eqref{eq:fkg} yields
	\begin{align*}
	    S(g+K\kappa)(x,a) &= WM(g + K \kappa)(x,a) = W(Mg + K \kappa)(x,a) \\
	    &\leq WMg(x,a) + W (K\kappa)(x,a) \leq Sg(x,a) + \alpha \beta K \kappa(x)
	\end{align*} 
	for all $x \in \XX$. 
	Since in addition $S$ is monotone, $S$ is a contraction mapping
	on $\cG_2$ as was to be shown. Moreover, because $(\cG_2, \|\cdot\|_\kappa)$
	is a Banach space, claims~\ref{item:rs.1}--\ref{item:rs.3} 
	follow immediately by the Banach contraction mapping theorem.
	
	Let $\cG_3$ be the set of upper semicontinuous functions in $\cG_2$.
	Similar to the proof of Theorem~\ref{t:cs}, we can show that 
	$\bar v \coloneqq Mg^*$ is a fixed point of $T$ in $\cV_2$ and $g^* = W \bar v$.
	A similar argument to the proof of Theorem~\ref{t:exist} shows that
	$\cG_3$ is a closed subset of $\cG_2$, $S$ maps $\cG_3$ into itself,
	the unique fixed point $g^*$ of $S$ is in $\cG_3$, and a $g$-greedy policy 
	exists for each $g \in \cG_3$. To see claim~\ref{item:rs.a} 
	holds, it remains to verify $v^* = \bar v$.
	
	To that end, we first show that \eqref{eq:Tnvbar} holds in the current
	setting. For each $\sigma \in \Sigma$ and $g \in \cG_2$, let $M_\sigma g$ 
	be defined as in \eqref{eq:Msig}. By the monotonicity of $M_\sigma$, 
	\eqref{eq:W_upbd} and \eqref{eq:fkg},
	\begin{align*}
		T_\sigma (\bar r + K \kappa) (x) 
		&= M_\sigma W (\bar r + K \kappa) (x) 
		\leq M_\sigma (W \bar r + W (K \kappa)) (x)  \\
		&= M_\sigma W \bar r(x) + W(K \kappa) (x, \sigma(x))
		\leq T_\sigma \bar r(x) + \alpha\beta K \kappa(x)
	\end{align*}
	for all $x \in \XX$ and $K \in \RR_+$. Induction shows that, for all 
	$x \in \XX$, $K \in \RR_+$ and $n \in \NN$,
	\begin{equation}
	\label{eq:Tsig_ineq}
		T_\sigma^n (\bar r + K \kappa) (x)
		\leq T_\sigma^n \bar r (x) + (\alpha\beta)^n K \kappa(x).
	\end{equation}
	Because $\bar v = M g^*$ and $g^*$ is in $\cG_2$, $g$ has a lower bound 
	$L \in \RR$ and 
	\begin{equation*}
	    \bar r + L \leq \bar v 
	    \leq \bar r + \|g^*\|_\kappa \kappa,
	\end{equation*}
	where both $\bar r$ and $\kappa$ are increasing functions on $\XX$. The
	monotonicity of $T_\sigma$, \eqref{eq:Tnsig} and \eqref{eq:Tsig_ineq} 
	then imply that, for all $x \in \XX$ and $n \in \NN$,
	\begin{align*}
	    T_\sigma^n \bar r(x) + (\alpha \beta)^n L
		\leq T_\sigma^n \bar v (x)
		\leq T_\sigma^n \bar r(x) + (\alpha \beta)^n \|g^*\|_\kappa \kappa (x).
	\end{align*}
	Letting $n \to \infty$ yields \eqref{eq:Tnvbar}. Similar to the proof of Theorem~\ref{t:rs0}, we can show that 
	\eqref{eq:vbar_ineq} holds in the current setting. Letting $n \to \infty$
	in \eqref{eq:vbar_ineq} then gives $\bar v \geq v_\sigma$ and thus 
	$\bar v \geq v^*$ since $\sigma$ is arbitrary. 
	A similar argument to the proof of Theorem~\ref{t:rs0} shows that
	$\bar v \leq v^*$. In summary, $\bar v = v^*$. Claim~\ref{item:rs.a} is
	verified.
	
	The proof of claims~\ref{item:rs.b}--\ref{item:rs.c} is same to the 
	proof of claims~\ref{item:rs0.b}--\ref{item:rs0.c} in Theorem~\ref{t:rs0} 
	and thus omitted. Therefore, all the statements of the theorem hold.
\end{proof}

\bibliographystyle{ecta}


\end{document}